\documentclass[12pt]{article}
\usepackage{fullpage,amsthm,amssymb,amsmath}
\usepackage{enumerate,enumitem}
\newtheorem{thm}{Theorem}[section]
\newtheorem{lemma}[thm]{Lemma}
\newtheorem{cor}[thm]{Corollary}
\newtheorem{rmk}[thm]{Remark}
\newtheorem{defn}[thm]{Definition}
\newtheorem{proposition}[thm]{Proposition}

\newtheorem{assumption}{Assumption}

\newcommand{\cL}{{\mathcal L}}
\newcommand{\te}{{\theta}}
\newcommand{\ve}{{\varepsilon}}
\newcommand{\cP}{{\mathcal P}}

\newcommand{\N}{\mathbb{N}}
\newcommand{\R}{\mathbb{R}}

\newcommand{\mcl}{\mathcal L}

\newcommand{\bbp}{\mathbb P}
\newcommand{\C}{\mathbb{C}}

\newcommand{\E}{\mathbb E}

\newcommand{\al}{\alpha}

\newcommand{\del}{\delta}

\newcommand{\ep}{\epsilon}
\newcommand{\eps}{\epsilon}
\newcommand{\sig}{\sigma}
\newcommand{\Sig}{\Sigma}
\newcommand{\ka}{\kappa}
\newcommand{\lam}{\lambda}
\newcommand{\Lam}{\Lambda}
\newcommand{\Om}{\Omega}
\newcommand{\om}{\omega}

\newcommand{\tld}[1]{\tilde{#1}}
\newcommand{\mc}[1]{\mathcal{#1}}

\DeclareMathOperator{\essinf}{ess\ inf}
\DeclareMathOperator{\esssup}{ess\ sup}
\DeclareMathOperator{\var}{var}

\newcommand{\bbC}{{\mathbb C}}

\newcommand{\bbR}{{\mathbb R}}

\newcommand{\B}{\mathcal{B}} 
\newcommand{\BV}{\B} 
\newcommand{\RR}{\mathcal R} 
\newcommand{\X}{X} 
\newcommand{\paeom}{\ensuremath{\bbp\text{-a.e. } \om \in \Om}}
\newcommand{\lot}{\lambda_\omega^\theta} 
\newcommand{\Lsot}{\mcl_{\sigma^{-1}\omega}^\theta} 

\begin{document}
\title{Limit theorems for random expanding or Anosov dynamical systems and vector-valued observables}

\author{Davor  Dragi\v cevi\' c \footnote{Department of Mathematics,
University of Rijeka,  Rijeka, Croatia. E-mail: \texttt{ddragicevic@math.uniri.hr}} and 
Yeor Hafouta \footnote{Department of Mathematics,
The Ohio State University,
Columbus OH USA. E-mail: \texttt{yeor.hafouta@mail.huji.ac.il, hafuta.1@osu.edu
}.} }

\maketitle

\begin{abstract}
The purpose of this paper is twofold. In one direction, we extend the spectral method for random piecewise expanding and hyperbolic (Anosov) dynamics developed by the first author \textit{et al}. to establish  quenched versions of the large deviation principle, central limit theorem and the local central limit theorem 
for \emph{vector-valued} observables. We stress that the previous works considered exclusively the case of scalar-valued observables.

In another direction, we show that this method can be used to establish a variety of new limit laws (either for scalar or vector-valued observables) that have not been discussed previously in the literature for the classes of dynamics we consider. More precisely, we establish the moderate deviations principle, 
concentration inequalities, Berry-Esseen estimates as well as Edgeworth and large deviation expansions.

Although our techniques rely on the approach developed in the previous works of the first author \textit{et al}., we emphasize that our arguments require several nontrivial adjustments as well as  new ideas. 
\end{abstract}

\section{Introduction}
The so-called spectral method represents a powerful approach for establishing limit theorems. It has been introduced by  Nagaev~\cite{N57, N61} in the context of Markov chains and by Guivarc’h and Hardy~\cite{GH88} as well as Rousseau-Egele~\cite{RE83} for the  deterministic dynamical systems.  We refer to~\cite{HennionHerve} for a detailed presentation of this method. 
In the case of deterministic dynamics, we have a map $T$ on the state space $X$ which preserves a probability measure $\mu$ on $X$. Then, for a suitable class of observables $g$, we want  to obtain limit laws for the process $(g\circ T^n)_{n\in \N}$. In other words, we wish to study the distribution of Birkhoff sums
$S_n g=\sum_{i=0}^{n-1}g\circ T^i$, $n\in \N$. 
 Let $\mcl$ be the transfer operator (acting on  a suitable Banach space $\mathcal B$) associated to $T$ and for each complex parameter $\theta$, let $\mcl^\theta$ be the so-called twisted transfer operator given by $\mcl^\theta f=\mcl (e^{\theta g}\cdot f)$, $f\in \mathcal B$. The core of the spectral method consists of the 
following steps:
\begin{itemize}
\item rewritting the characteristic function of $S_n g$ in terms of the powers of the twisted transfer operators $\mcl^\theta$;
\item applying the classical Kato's perturbation theory  to show that for  $\theta$ sufficiently close to $0$, $\mcl^\theta$ inherits nice spectral properties from $\mcl$. More precisely, usually one works under assumptions which ensure that $\mcl$ is a quasi-compact operator of spectral radius $1$ with the property that $1$ is the only eigenvalue 
on the unit circle with multiplicity one (and  with the eigenspace essentially corresponding to  $\mu$). Then, for $\theta$ sufficiently close to $0$, $\mcl^\theta$ is again a quasi-compact operator with an isolated eigenvalue of multiplicity one such that both the eigenvalue and the corresponding 
eigenspace (as well as other related objects) depend analytically on $\theta$. 
\end{itemize}
This method has been used to establish a variety of limit laws for broad classes of chaotic deterministic systems exhibiting some degree of hyperbolicity. Indeed, it has been used to establish   large deviation principles \cite{HennionHerve,rey-young08}, central limit theorems \cite{RE83,Broise,HennionHerve,AyyerLiveraniStenlund}, Berry-Esseen bounds \cite{GH88,gouezel05}, local central limit theorems \cite{RE83,HennionHerve,gouezel05} as well as the almost sure invariance principle~\cite{gouezel10}.  We refer to the excellent survey paper~\cite{G15} for more details and further references. 

Very recently, the spectral method was extended to broad classes of random dynamical systems. More precisely, the first author \textit{et al}. adapted the spectral method in order to obtain several quenched limit theorems for random piecewise expanding as well as random hyperbolic dynamics~\cite{DFGTV1, DFGTV2}. 
In particular they proved the first version of the quenched local central limit in the context of random dynamics. A similar task was independently accomplished for random distance expanding dynamics by the second author and Kifer~\cite{HK}.  We stress that the study of the statistical properties of the random or time-dependent dynamical systems was initiated by Bakhtin~\cite{B1, B2} and 
 Kifer~\cite{K96,K98} using different techniques from those in~\cite{DFGTV1, DFGTV2} (and the present paper).  Indeed, the methods in~\cite{B1, B2} rely on the use of real Birkhoff cones (and share some similarities with the approach in~\cite{HK}) although Bakhtin does not discuss the local central limit theorem and the dynamics he considered does not allow the presence of singularities.  Moreover, his results do not include  the large deviations principles obtained in~\cite{DFGTV1, DFGTV2}. 
On the other hand, all the results in~\cite{K98} rely on the martingale method which although also very powerful,  cannot for example  be used to obtain a local central limit theorem.

Let us now briefly discuss the main ideas from~\cite{DFGTV1, DFGTV2, HK}. Instead of a single map as in the deterministic setting, we now have a collection of maps $(T_\om)_{\om \in \Om}$ acting on a state space $X$, where $(\Omega, \mathcal F, \mathbb P)$ is a probability space. We consider random compositions of the form
\[
T_\om^{(n)}=T_{\sigma^{n-1}\om}\circ \ldots \circ T_\om \quad \text{for $\om \in \Om$ and $n\in \N$,}
\]
where $\sigma \colon \Om \to \Om$ is an invertible $\mathbb P$-preserving transformation. Under appropriate conditions, there exists a unique family of probability measures $(\mu_\om)_{\om \in \Om}$ on $X$ such that $T_\om^*\mu_\om=\mu_{\sigma \om}$ for $\mathbb P$-a.e. $\om \in \Om$. Then, for a suitable class
of observables $g\colon \Om \times X\to \R$, we wish to establish limit laws for the process $(g_{\sigma^n \om }\circ T_\om^{(n)})_{n\in \N}$ with respect to $\mu_\om$, where $g_\om:=g(\om, \cdot)$, $\om \in \Om$.  Let $\mcl_\om$ denote the transfer operator associated to $T_\om$ (acting on a suitable Banach space $\mathcal B$). In a similar manner to that
in the deterministic case, for each $\theta \in \C$ and $\om \in \Om$ we consider the twisted transfer operator $\mcl_\om^\theta$ on $\mathcal B$ defined by $\mcl_\om^\theta f=\mcl (e^{\theta g(\om, \cdot)}f)$, $f\in \mathcal B$. Then, the arguments in~\cite{DFGTV1, DFGTV2} proceed as follows:
\begin{itemize}
\item we represent the characteristic functions of the random Birkhoff sums 
\begin{eqnarray*}
&S_n g(\om, \cdot)=\sum_{i=0}^{n-1} g_{\sigma^i \omega}(T_\om^{(i)}(\cdot))
\end{eqnarray*} 
in terms of twisted transfer operators;
\item in the language of the multplicative ergodic theory, for $\theta$ sufficiently close to $0$, the twisted cocycle $(\mcl_\om^\theta)_{\om \in \Om}$ is quasi-compact, its largest Lyapunov exponents has mulitiplicity one (i.e. the associated Oseledets subspace is one-dimensional) and similarly to
 the deterministic case all these objects exhibit sufficiently regular behaviour with respect to $\theta$.
\end{itemize}
Although Lyapunov exponents and associated Oseledets subspaces precisely represent a nonautonomous analogons of eigenvalues and eigenspaces, we emphasize that the methods in~\cite{DFGTV1, DFGTV2} require a highly nontrivial adjustments of the classical spectral method for deterministic dynamics. 

The goal of the present paper is twofold. In one direction, we wish to extend the main results from~\cite{DFGTV1, DFGTV2} by establishing quenched versions of  the large deviations principle, central limit theorem and the local central limit for vector-valued observables. We stress that in~\cite{DFGTV1, DFGTV2} the authors dealt only with scalar-valued observables. Although in order to accomplish this we heavily rely on the previous work, we stress that the treatment of vector-valued observables requires several changes of nontrivial nature when compared to the previous papers. 

In another direction, we show that the spectral method developed in~\cite{DFGTV1, DFGTV2} can be used to establish a variety of new limit laws (either for scalar or vector-valued observables) that have not been considered previously in the literature (at least for the classes of dynamics that are considered in the present paper).
Indeed, we here for the first time discuss a moderate deviations principle, Berry-Esseen bounds, concentration inequalities, Edgeworth and certain large deviations expansions for random piecewise expanding and hyperbolic dynamics. We emphasize that each of these results requires nontrivial adaptation of the techniques 
developed in~\cite{DFGTV1, DFGTV2}. We in particularly stress that similarly to~\cite{DFGTV1, DFGTV2},  none of our results require any mixing assumptions for the base map $\sigma$.

Finally, we would like to briefly mention some of other works devoted to statistical properties of random dynamical systems. We particularly mention the works of Ayyer, Liverani and Stenlund~\cite{AyyerLiveraniStenlund} as well as 
Aimino, Nicol and~Vaienti~\cite{ANV15} that preceded~\cite{DFGTV1}. They also discuss limit laws for random toral automorphisms and random piecewise expanding 
maps respectively  but under a restrictive assumption that the base space $(\Om, \sig)$ is a Bernoulli shift. Furthermore, we mention the recent interesting papers by Bahsoun and collaborators~\cite{ABR, BahsounBose, BBR} as well as Su~\cite{Su} concerned with the decay of correlation and 
limit laws for systems which can be modelled by random Young towers. Further relevant contributions to the study of statistical properties of random or time-dependent dynamics have been established by N{\'a}ndori, Sz{\'a}sz, and Varj{\'u}~\cite{NandoriSzaszVarju}, Nicol, T\"{o}r\"{o}k and Vaienti~\cite{nicol_torok_vaienti_2016},
Hella and Stenlund~\cite{HeSt}, Lepp{\"a}nen and Stenlund~\cite{LeppanenStenlund, LeppanenStenlund2} as well as the second author~\cite{HafEDG,HafSDS}. We also refer the readers to corresponding results for inhomogeneous Markov chains, including ones arising as almost sure realizations of Markov chains in random (dynamical) environments due to Dolgopyat and Sarig \cite{DS} and Kifer and the second author \cite{HK}.

\section{Preliminaries}
In this section we recall basic notions and results from the multiplicative ergodic theory which will be used in the subsequent sections. The material is essentially taken from~\cite{DFGTV1} but we include it for readers' convenience.
\subsection{Multiplicative ergodic theorem}\label{sectionMET}
In this subsection we recall the recently established versions of the multiplicative ergodic theorem which can be applied to the study of cocycles of transfer operators and  will play
an important role in the present paper. We begin by recalling some basic notions.

A tuple $\mc{R}=(\Om, \mc{F}, \bbp, \sigma, \B, \mcl)$ will be called a linear cocycle, or simply a\ \textit{cocycle}, if
 $\sigma$ is an invertible ergodic measure-preserving transformation on a probability space $(\Omega,\mathcal F,\mathbb P)$,  $(\B, \|\cdot \|)$ is a Banach space and
 $\mathcal L\colon \Omega\to L(\B)$ is a family of bounded linear operators such that
$\log^+\|\mathcal L(\omega)\|\in L^1(\mathbb P)$. 
Sometimes we will also use $\mc{L}$ to refer to the full cocycle $\mc{R}$.
In order to obtain sufficient measurability conditions, we assume the following:
\begin{enumerate}[label=(C\arabic*), series=conditions]
\setcounter{enumi}{-1}
\item \label{cond:METCond}
 $\Om$ is a Borel subset of a separable, complete metric space, $\sigma$ is a homeomorphism and $\mcl$ is either  $\bbp-$continuous (that is, $\mcl$ is  continuous on each of countably many Borel sets whose union is $\Omega$) or strongly measurable (that is, the map $\omega \mapsto \mcl_\omega f$ is measurable for each $f\in \B$) and $\B$ is separable.
\end{enumerate}
For each $\om \in \Om$ and $n\geq 0$,  let $ \mcl_\om^{(n)}$ be the linear operator given by  \[ \mcl_\om^{(n)}:= \mcl_{\sig^{n-1}\om}\circ \dots \circ\mcl_{\sig\om} \circ \mcl_\om. \]
Condition~\ref{cond:METCond} implies that the map $\om \mapsto \log \| \mcl_\om^{(n)}\|$ is measurable for each $n\in \N$. Thus,
Kingman's sub-additive ergodic theorem ensures that the following limits exist and coincide for \paeom:
\begin{align*}
\Lam(\mc{R}) &:= \lim_{n\to \infty} \frac1n\log \| \mcl_\om^{(n)}\|\\
\ka(\mc{R}) &:=  \lim_{n\to \infty} \frac1n\log \text{ic}( \mcl_\om^{(n)}),
\end{align*}
where \[\text{ic}(A):=\inf\Big\{r>0 : \ A(B_\B) \text{ can be covered with finitely many balls of radius }r \Big\},\] and $B_\B$ is the unit ball of $\B$.
 The cocycle $\mc{R}$ is called \textit{quasi-compact} if
$\Lam(\mc{R})> \ka(\mc{R})$.
 The quantity $\Lam(\mc{R})$ is called the \textit{top Lyapunov exponent} of the cocycle and generalises the notion of (logarithm of) spectral radius of a linear operator. Furthermore,  $\ka(\mc{R})$ generalises the notion of essential spectral radius to the context of cocycles.
\begin{rmk}
We refer to~\cite[Lemma 2.1]{DFGTV1} for useful criteria which can be used to verify that the cocycle is quasi-compact.
\end{rmk}
 A spectral-type decomposition for quasi-compact cocycles can be obtained via  the following \textit{multiplicative ergodic theorem}.
\begin{thm}[Multiplicative ergodic theorem, MET \cite{BlumenthalMET,GTQuas1,FLQ2}] \label{thm:MET}
Let $\mathcal R=(\Omega,\mathcal F,\mathbb
P,\sigma,\B,\mathcal L)$ be a  quasi-compact cocycle and
suppose that condition~\ref{cond:METCond} holds.
Then, there exists $1\le l\le \infty$ and a sequence of exceptional
Lyapunov exponents
\[ \Lam(\RR)=\lambda_1>\lambda_2>\ldots>\lambda_l>\kappa(\RR) \quad \text{(if $1\le l<\infty$)}\]
or  \[ \Lam(\RR)=\lambda_1>\lambda_2>\ldots \quad \text{and} \quad \lim_{n\to\infty} \lambda_n=\kappa(\RR) \quad \text{(if $l=\infty$);} \]
 and for $\mathbb P$-a.e. $\omega \in \Om$ there exists a unique splitting (called the \textit{Oseledets splitting}) of $\B$ into closed subspaces
\begin{equation}\label{eq:splitting}
\B=V(\omega)\oplus\bigoplus_{j=1}^l Y_j(\omega),
\end{equation}
depending measurably on $\om$ and  such that:
\begin{enumerate} [label=(\Roman*)]
\item  For each $1\leq j \leq l$, $Y_j(\omega)$ is finite-dimensional ($m_j:=\dim Y_j(\omega)<\infty$),  $Y_j$ is equivariant i.e. $\mcl_\om Y_j(\omega)= Y_j(\sigma\omega)$ and for every $y\in Y_j(\omega)\setminus\{0\}$,
\[\lim_{n\to\infty}\frac 1n\log\|\mathcal L_\omega^{(n)}y\|=\lambda_j.\]
(Throughout this paper, we will  also refer to $Y_1(\om)$ as simply $Y(\om)$ or $Y_\om$.)
\item
$V$ is equivariant i.e. $\mcl_\om V(\omega)\subseteq V(\sigma\omega)$ and
for every $v\in V(\omega)$, \[\lim_{n\to\infty}\frac 1n\log\|\mathcal
L_\omega^{(n)}v\|\le \kappa(\RR).\]
\end{enumerate}
\end{thm}
The \textit{adjoint cocycle} associated to $\mc{R}$ is the cocycle  $\mc{R}^*:=(\Om, \mc{F}, \bbp, \sigma^{-1}, \B^*, \mcl^*)$, where $(\mcl^*)_\om := (\mcl_{\sig^{-1}\om})^*$.
In a slight abuse of notation which should not cause confusion, we will often write $\mcl^*_\om$ instead of $(\mcl^*)_\om$, so $\mcl^*_\om$ will denote the operator adjoint to $\mcl_{\sig^{-1}\om}$.

The following two results are taken from~\cite{DFGTV1}.
\begin{cor}\label{cor:METAdjoint}
Under the assumptions of Theorem~\ref{thm:MET}, the adjoint cocycle $\mc{R}^*$ has a unique, measurable, equivariant Oseledets splitting
\begin{equation}\label{eq:splitting_adj}
\B^*=V^*(\omega)\oplus\bigoplus_{j=1}^l Y^*_j(\omega),
\end{equation}
with the same exceptional Lyapunov exponents $\lam_j$ and multiplicities $m_j$ as $\mc{R}$.
\end{cor}

Let the simplified Oseledets decomposition for the cocycle $\mcl$ (resp. $\mcl^*$) be
\begin{equation}\label{BVstar}
\mathcal B=Y(\om)\oplus H(\om) \quad
(\text{resp. } \mathcal{B}^*=Y^*(\om) \oplus H^*(\om) ),
 \end{equation}
where $Y(\om)$ (resp. $Y^*(\om)$) is the top Oseledets
subspace for $\mcl$ (resp. $\mcl^*$) and $H(\om)$ (resp. $H^*(\om)$) is a direct sum of all other Oseledets subspaces.

For a subspace $S\subset \mathcal B$, we set $ S^\circ=\{\phi \in \mathcal{B}^*: \phi(f)=0 \quad \text{for every $f\in S$}\}$
and similarly for a subspace $S^* \subset \mathcal{B}^*$ we define
$
 (S^*)^\circ =\{f\in \mathcal{B}: \phi(f)=0 \quad \text{for every $\phi \in S^*$}\}.
$
\begin{lemma}[Relation between Oseledets splittings of $\mc{R}$ and $\mc{R}^*$]\label{lem:AnnihilatorOsSplittings}
The following relations hold for \paeom:
 \begin{equation}\label{jj}
  H^*(\om)=Y(\om)^\circ \quad \text{and} \quad H(\om)=Y^*(\om)^\circ.
 \end{equation}

\end{lemma}

\section{Piecewise-expanding dynamics}
In this section we introduce the class of random piecewise expanding dynamics we plan to study (which is the same as considered in~\cite{DFGTV1}). We then proceed by introducing a class of vector-valued observables to which our limit theorems will apply. Furthermore, for $\theta \in \C^d$,  we introduce the corresponding  twisted cocycle
of transfer operators $(\mcl_\om^\theta)_{\om \in \Om}$. Finally, we study the regularity (with respect to $\theta$) of the largest Lyapunov exponent and the corresponding top Oseledets space of the cocycle $(\mcl_\om^\theta)_{\om \in \Om}$. Our arguments in this section follow closely the approach developed in~\cite{DFGTV1}. We refer as much as possible to~\cite{DFGTV1}, discussing in detail  only the arguments which require substantial changes (when compared to~\cite{DFGTV1}).

\subsection{Notions of variation}\label{sec:var}
Let $(X, \mathcal G)$ be a measurable space endowed with a probability measure $m$ and a notion of a variation $\var \colon L^1(X, m) \to [0, \infty]$ which satisfies
the following conditions:
\begin{enumerate}
 \item[(V1)] $\var (th)=\lvert t\rvert \var (h)$;
 \item[(V2)] $\var (g+h)\le \var (g)+\var (h)$;
 \item[(V3)] $\lVert h\rVert_{L^\infty} \le C_{\var}(\lVert h\rVert_1+\var (h))$ for some constant $1\le C_{\var}<\infty$;
 \item[(V4)] for any $C>0$, the set  $\{h\colon X \to \mathbb R: \lVert h\rVert_1+\var (h) \le C\}$ is $L^1(m)$-compact;
 \item[(V5)] $\var(1_X) <\infty$, where $1_X$ denotes the function equal to $1$ on $X$;
 \item[(V6)] $\{h \colon X \to \mathbb R_+: \lVert h\rVert_1=1 \ \text{and} \ \var (h)<\infty\}$ is $L^1(m)$-dense in
 $\{h\colon X \to \mathbb R_+: \lVert h\rVert_1=1\}$.
 \item[(V7)] for any $f\in L^1(X, m)$ such that $\essinf f>0$, we have $\var(1/f) \le \frac{\var (f)}{(\essinf f)^2}$.
 \item[(V8)] $\var (fg)\le \lVert f\rVert_{L^\infty}\cdot \var(g)+\lVert g\rVert_{L^\infty}\cdot \var(f)$.
 \item[(V9)] for $M>0$, $f\colon X \to \overline{B}_{\R^d} (0, M)$ measurable and  every $C^1$ function $h\colon \overline B_{\R^d} (0, M) \to \C$, we have
 $\var (h\circ f)\le \sup \{ \lVert Dh(P)\rVert : P\in \overline B_{\R^d}(0, M) \} \cdot \var(f)$. Here, $\overline B_{\R^d}(0, M)$ denotes the closed ball in $\R^d$ centered in $0$ with radius $M$.
\end{enumerate} We define
\[
 \B:=BV=BV(X,m)=\{g\in L^1(X, m): \var (g)<\infty \}.
\]
Then, $\B$ is a Banach space with respect to the norm
\[
 \lVert g\rVert_{\B} =\lVert g\rVert_1+ \var (g).
 \]
From now on, in this section,  we will use $\B$ to denote a Banach space of this type, and $ \lVert g\rVert_{\B} $, or simply $\|g\|$ will denote the corresponding norm.

We note that examples of this notion correspond to the case where $X$ is a  subset of $\R^n$.
In the one-dimensional case we use the classical notion of variation given by
\begin{equation}\label{var1d}
 \var (g)=\inf_{h=g (mod \ m)} \sup_{0=s_0<s_1<\ldots <s_n=1}\sum_{k=1}^n \lvert h(s_k)-h(s_{k-1})\rvert
\end{equation}
for which it is well known that properties (V1)-(V9) hold.
On the other hand, in the multidimensional case (see~\cite{Saussol}), we let $m=Leb$ and define
\begin{equation}\label{varmd}
 \var (f)=\sup_{0<\epsilon \le \epsilon_0}\frac{1}{\epsilon^\alpha}\int_{\R^d}\text{osc} (f,  B_\epsilon (x)))\, dx,
\end{equation}
where
\[
\text{osc} (f, B_\epsilon (x))=\esssup_{x_1, x_2 \in B_\epsilon (x)}\lvert f(x_1)-f(x_2)\rvert
\]
and where $\esssup$ is taken with respect to product measure $m\times m$. It has been discussed in~\cite{DFGTV1} that in this case $\text{var}(\cdot)$ again satisfies properties (V1)-(V9).

In another direction, by  taking $\text{var}(\cdot)$ to be a H\"{o}lder constant and $X$ to be a compact metric space, our framework also includes distance expanding maps considered in~\cite{HK} and~\cite{MSU} which are non-singular with respect to a given measure $m$ (in particular we consider the case of identical fiber spaces $X_\om=X$).



\subsection{A cocycles of transfer operators}\label{COTO}

Let  $(\Omega, \mathcal{F}, \mathbb P, \sigma)$ be  as Section~\ref{sectionMET}, and $X$ and $\BV$ as in Section~\ref{sec:var}.
Let  $T_{\omega} \colon X \to X$, $\omega \in \Omega$ be a collection of non-singular  transformations (i.e.\ $m\circ T_\omega^{-1}\ll m$ for each $\omega$) acting   on $X$.
The associated skew product transformation  $\tau \colon  \Omega \times X \to  \Omega \times X$ is defined by
\begin{equation}
\label{eq:tau}
\tau(\omega, x)=( \sigma (\omega), T_{\omega}(x)), \quad \om \in \Om, \ x\in X.
\end{equation}
 Each transformation $T_{\omega}$ induces the corresponding transfer operator $\mathcal L_{\omega}$ acting on $L^1(X, m)$ and  defined  by the following duality relation
\[
\int_X(\mathcal L_{\omega} \phi)\psi \, dm=\int_X\phi(\psi \circ T_{\omega})\, dm, \quad \phi \in L^1(X, m), \ \psi \in L^\infty(X, m).
\]
For each $n\in \mathbb N$ and $\omega \in \Omega$, set
\[
T_{\omega}^{(n)}=T_{\sigma^{n-1} \omega} \circ \cdots \circ T_{\omega} \quad \text{and} \quad \mathcal L_{\omega}^{(n)}=\mathcal L_{\sigma^{n-1} \omega} \circ \cdots \circ \mathcal L_{\omega}.\]

\begin{defn}[Admissible cocycle] \label{def:admis}
We call the transfer operator cocycle $\mc{R}=(\Om, \mathcal F, \bbp, \sig, \B, \mathcal L)$ admissible if the following conditions hold:
\begin{enumerate}[label=(C\arabic*), series=conditions]
\item $\mc{R}$ is $\mathbb P$-continuous (that is, $\mathcal L$ is continuous in $\om$ on each of countably many Borel
sets whose union is $\Om$);
\item \label{cond:unifNormBd}
there exists $K>0$ such that
\begin{equation*}
 \lVert \mathcal L_\om f\rVert_{\BV} \le K\lVert f\rVert_{\BV}, \quad \text{for every $f\in \BV$ and $\paeom$.}
\end{equation*}

\item \label{C0}
there exists $N\in \mathbb N$ and measurable $\alpha^N, \beta^N \colon \Om \to (0, \infty)$, with
$ \int_\Om \log \alpha^N (\om)\,  d\mathbb P(\om)<0$,
 such that for every $f\in \BV$ and $\paeom$,
\begin{equation*}
 \lVert \mcl_\om^{(N)} f\rVert_{\BV} \le \alpha^N(\om)\lVert f\rVert_{\BV}+\beta^N(\om)\lVert f\rVert_1.
\end{equation*}
\item  \label{cond:dec} there exist $K', \lambda >0$ such that
for every $n\ge 0$, $f\in \BV$ such that $\int f\, dm=0$ and $\paeom$.
\begin{equation*}
\lVert \mathcal L_{\om}^{(n)} (f)\rVert_{\BV} \le K'e^{-\lambda n}\lVert f\rVert_{\BV}.
\end{equation*}

\item  \label{Min} there exist $N\in \N, c>0$ such that for each $a>0$  and any sufficiently large $n\in \mathbb N$,
\begin{equation*}
\essinf  \mathcal L_\omega^{(Nn)} f\ge c \lVert f\rVert_1, \quad \text{for every $f\in C_a$ and \paeom,}
\end{equation*}
where $C_a:=\{ f \in \BV: f\ge 0 \text{ and } \var(f)\le a\int f\, dm \}.$
\end{enumerate}
\end{defn}

\begin{rmk}
We note that we have imposed condition (C1) since in this setting  $\B$ is not separable. 
\end{rmk}

\begin{rmk}
We refer to~\cite[Section 2.3.1]{DFGTV1} for explicit examples of admissible cocycles of transfer operators associated to piecewise expanding maps both in dimension $1$ as well as in higher dimensions.  
\end{rmk}

The following result is established in~\cite[Lemma 2.9.]{DFGTV1}.
\begin{lemma}\label{lem:qc+1dim}
 An admissible cocycle of transfer operators $\mathcal R=(\Om, \mathcal F, \mathbb P, \sig, \BV, \mcl)$  is quasi-compact.
 Furthermore, the top Oseledets space is one-dimensional. That is, $\dim Y(\om)=1$ for \paeom.
\end{lemma}

The following result established in~\cite[Lemma 2.10.]{DFGTV1} shows that, in this context, the top Oseledets space is  spanned by the unique random absolutely continuous invariant measure (a.c.i.m. for short).
We recall that random a.c.i.m. is a measurable map $v^0: \Om \times \X \to \R^+$ such that
for \paeom, $v^0_\om:= v^0(\om, \cdot) \in \B$, $\int v^0_\om(x)dm=1$ and
\begin{equation}\label{v0om}
 \mathcal L_\om v_\om^0=v_{\sigma \om}^0, \quad \text{for $\paeom$.}
\end{equation}

\begin{lemma}[Existence and uniqueness of a random acim]\label{lem:PF}
Let $\mathcal R=(\Omega,\mathcal F,\mathbb
P,\sigma,\B,\mathcal L)$ be an admissible cocycle of transfer operators.
Then, there exists a unique random absolutely continuous invariant measure
for $\mathcal R$.
\end{lemma}

For an admissible transfer operator cocycle $\mc{R}$, we let $\mu$ be the invariant probability measure given by
\begin{equation}\label{eq:defmu}
 \mu(A \times B)=\int_{A\times B} v^0(\om, x)\, d (\mathbb P \times m)(\om, x), \quad \text{for $A\in \mathcal F$ and $B\in \mathcal G$,}
\end{equation}
where $v^0$ is the unique random a.c.i.m. for $\mc{R}$ and $\mc{G}$ is the Borel $\sigma$-algebra of $X$.
We note that $\mu$ is $\tau$-invariant, because 
of \eqref{v0om}. Furthermore, for each $G\in L^1(\Omega \times X, \mu)$ we have that
\[
 \int_{\Omega \times X} G\, d\mu=\int_{\Omega} \int_X G(\om, x)\, d\mu_\om(x)\, d\mathbb P(\om),
\]
where $\mu_\om$ is a measure on $X$ given by $d\mu_\om=v^0(\om, \cdot)dm$.

Let us recall the following result established in~\cite[Lemma 2.11.]{DFGTV1}.
\begin{lemma}\label{lem:boundedv}
The unique random a.c.i.m. $v^0$ of an  admissible cocycle of transfer operators satiesfies the following:
\begin{enumerate}
\item  \label{it:boundedv}
\begin{equation}\label{eq:boundedv}
  \esssup_{\om \in \Om} \lVert v_\om^0\rVert_{\BV} <\infty;
 \end{equation}
\item  there exists $c>0$ such that \begin{equation}\label{lowerbound} \essinf v_\om^0 (\cdot)\ge c, \quad \text{for  $\paeom$;}
\end{equation}
\item  there exists $K>0$ and $\rho \in (0, 1)$  such that
 \begin{equation}\label{buzzi}
 \bigg{\lvert} \int_X \mathcal L_\omega^{(n)}(f v_\omega^0)h\, dm -\int_X f \, d\mu_\omega \cdot \int_X h \, d\mu_{\sigma^n \omega} \bigg{\rvert} \le K\rho^n
 \lVert h\rVert_{L^\infty} \cdot \lVert f \rVert_{\BV} ,
\end{equation}
for $n\ge 0$, $h \in L^\infty (X, m)$, $f \in \BV$ and $\paeom$.
\end{enumerate}
\end{lemma}

\subsection{The observable}\label{sec:obs}
Let us now introduce a class of observables to which our limit theorems will apply (although in some cases we will require addtional assumptions).
\begin{defn}[Observable]\label{def:obs}
An \emph{observable} is a
 measurable map $g \colon \Omega \times X \to \mathbb R^d$, $g=(g^1, \ldots, g^d)$ satisfying the following properties:
 \begin{itemize}
 \item Regularity:
\begin{equation}\label{obs}
 \lVert g(\om, x)\rVert_{L^\infty (\Omega \times X)}=: M<\infty \quad \text{and} \quad  \esssup_{\om \in \Om} \var (g_\om) <\infty,
\end{equation}
where $g_\omega=g (\omega, \cdot )$ and $\var (g_\om):=\max_{1\le i\le d}\var (g_\om^i)$,  $\omega \in \Omega$.
\item Fiberwise centering:
\begin{equation}\label{zeromean}
 \int_X g^i(\om, x) \, d\mu_\om(x)= \int_X g^i(\om, x)v^0_\om (x) \, dm(x)=0 \quad \text{for $\mathbb P$-a.e. $\om \in \Om$, $1\le i\le d$,}
\end{equation}
where $v^0$ is the density of the unique random a.c.i.m., satisfying \eqref{v0om}.
\end{itemize}
\end{defn}
\begin{rmk}
The class of observables considered in~\cite{DFGTV1} are scalar-valued, i.e. correspond to the case when $d=1$.
\end{rmk}

We also introduce the corresponding random Birkhoff sums. More precisely, for $n\in \N$ and $(\om, x)\in \Om \times X$, set
\[
S_n g(\omega, x):=\sum_{i=0}^{n-1}g(\sigma^i \om, T_\omega^{(i)}(x)).
\]
\subsection{Basic properties of twisted transfer operator cocycles}
Throughout this section, $\mc{R}=(\Om, \mc{F}, \bbp, \sig, \B, \mcl)$ will denote an admissible transfer operator cocycle. Furthermore, by $x\cdot y$ we will denote the scalar product of $x, y\in \mathbb C^d$ and $\lvert x\rvert$ will denote the norm of $x$.

For an observable $g$ as in Definition~\ref{def:obs} and  $\theta \in \mathbb C^d$,
the \emph{twisted transfer operator cocycle} (or simply a \emph{twisted cocycle}) $\mc{R}^\theta$  is defined as
 $\mc{R}^\theta=(\Om, \mc{F},  \bbp, \sig, \BV, \mcl^\theta)$,
where   for each $\omega \in \Omega$, we define
\begin{equation}\label{eq:twisted}
\mathcal L_\omega^{\theta}(f)=\mathcal L_\omega (e^{\theta \cdot g(\omega, \cdot )}f), \quad f\in \BV.
\end{equation}
 For convenience of notation, we will also use $ \mcl^\theta$ to denote the cocycle $\mc{R}^\theta$.
For each $\theta \in \C^d$, set $\Lam(\theta):=\Lam(\mc{R}^\theta)$  and
\[
\mathcal L_\om^{\theta, \, (n)}=\mathcal{L}^{\theta}_{\sigma^{n-1}\omega}\circ\cdots\circ \mathcal{L}^{\theta}_\omega, \quad \text{for $\om \in \Om$ and $n\in \N$.}
\]
\begin{lemma}\label{lj}
For $\mathbb P$-a.e. $\om \in \Om$ and $\theta \in \mathbb C^d$, 
\[
\var (e^{\theta \cdot g(\om, \cdot)}) \le \lvert \theta \rvert  e^{\lvert \theta \rvert M}\var (g(\om, \cdot)).
\]
\end{lemma}

\begin{proof}
The conclusion of the lemma follows directly from $(V9)$ applied for $f=g(\om, \cdot)$ and $h$ given by $h(z)=e^{\theta \cdot z}$ by  taking into account~\eqref{obs}.
\end{proof}
\begin{lemma}\label{l49}
 There exists a continuous function $K\colon \mathbb C^d \to (0, \infty)$ such that
 \begin{equation}\label{se2}
 \lVert \mathcal L_\om^\theta h\rVert_{\BV} \le K(\theta)\lVert h\rVert_{\BV}, \quad \text{for $h\in \BV$, $\theta \in \mathbb C$ and $\paeom$.}
\end{equation}
\end{lemma}

\begin{proof}
It follows from~\eqref{obs}   that for any $h\in \B$,  $\lvert e^{\theta \cdot g(\om, \cdot)}h\rvert_1 \le e^{\lvert \theta \rvert M}\lvert h\rvert_1$. Furthermore, (V8) implies that
\[
\var( e^{\theta \cdot g(\om, \cdot)}h )\le \lVert e^{\theta \cdot g(\om, \cdot)}\rVert_{L^\infty}\cdot \var(h)+\var(e^{\theta \cdot g(\om, \cdot)})\cdot \lVert h\rVert_{L^\infty},
\]
which together with (V3) and  Lemma~\ref{lj}  yields that
\[
\begin{split}
\lVert  e^{\theta \cdot g(\om, \cdot)}h  \rVert_{\B} &=\var  (e^{\theta \cdot g(\om, \cdot)}h) +\lvert e^{\theta \cdot g(\om, \cdot)}h \rvert_1 \\
&\le e^{\lvert \theta \rvert M}\lVert h\rVert_{\B}+
 +\lvert \theta \rvert e^{\lvert \theta \rvert M} \esssup_{\om \in \Om}\var (g(\om, \cdot)) \lVert h\rVert_{L^\infty} \\
&\le (e^{\lvert \theta \rvert M}+C_{\var}\lvert \theta \rvert e^{\lvert \theta \rvert M} \esssup_{\om \in \Om}\var (g(\om, \cdot)))\lVert h\rVert_{\B}.
\end{split}
\]
Thus, from (C2) we conclude that~\eqref{se2} holds with
\[
K(\theta)=K\left(e^{\lvert \theta \rvert M}+C_{\var}\lvert \theta \rvert e^{\lvert \theta \rvert M} \esssup_{\om \in \Om}\var (g(\om, \cdot))\right).
\]
\end{proof}

\begin{lemma}\label{lem:exprLit}
The following statements hold:
\begin{enumerate}
\item for every $\phi \in \BV^*, f \in \BV$, $\om \in \Om$,  $\theta \in \mathbb C^d$ and $n\in \N$ we have that
\begin{equation}\label{eq:adjointPert}
\mcl_\om^{\theta, (n)}(f)=\mcl_\om^{(n)}(e^{\theta \cdot S_{n}g(\om, \cdot)}f), \quad \text{and} \quad
 \mcl_\om^{\theta*,(n)}(\phi)  =  e^{\theta \cdot S_ng(\om, \cdot)}  \mcl_\om^{*(n)}(\phi),
\end{equation}
where $(e^{\theta \cdot S_ng(\om, \cdot)} \phi) (f):= \phi (e^{\theta \cdot  S_ng(\om, \cdot)} f)$;
\item for every $f\in \BV$, $\om \in \Om$ and $n\in \N$ we have that
\begin{equation}\label{intprop}\int_X \mathcal{L}^{\theta, \, (n)}_\omega (f)\ dm=\int_X e^{\theta \cdot S_ng(\omega, \cdot )}f\ dm. \end{equation}
\end{enumerate}
\end{lemma}

\begin{proof}

We establish the first identity in~\eqref{eq:adjointPert} by induction on $n$. The case $n=1$ follows from the definition of $\mcl_\om^{\theta}$.
We recall that for every $f, \tilde{f}\in \BV$,
\begin{equation}\label{eq:prodRuleP}
\mcl_\om^{(n)}((\tilde{f} \circ T_\om^{(n)})  \cdot f) = \tilde{f}\cdot  \mcl_\om^{(n)}( f).
\end{equation}
Let us assume that the claim holds for some $n$. Then, using~\eqref{eq:prodRuleP} we have that  
\begin{align*}
 \mcl_{\om}^{(n+1)} (e^{\theta \cdot S_{n+1}g(\om, \cdot)}f)
 &= \mcl_{\sig^n\om} \big(\mcl_{\om}^{(n)} ( e^{\theta \cdot  g(\sig^{n} \om, \cdot)\circ T_\om^{(n)}}  e^{\theta \cdot S_{n}g(\om, \cdot)}f) \big) \\
 &=
 \mcl_{\sig^n\om} \big( e^{\theta \cdot g(\sig^{n} \om, \cdot)} \mcl_{\om}^{(n)} ( e^{\theta  \cdot S_{n}g(\om, \cdot)}f) \big)
 = \mcl_{\sig^n\om}^{\theta }\mcl_\om^{\theta , (n)}(f) = \mcl_\om^{\theta , (n+1)}(f).
\end{align*}
The second identity in~\eqref{eq:adjointPert}  follows directly from duality. Finally, \eqref{intprop} follows by integrating the first equality in~\eqref{eq:adjointPert}.
\end{proof}

\subsection{An auxiliary existence and regularity result}\label{lar}
We now recall the construction of Banach spaces introduced in~\cite{DFGTV1} that play an important role in the spectral analysis of the twisted cocycle.

Let $\mc{S}'$ denote the set of all measurable functions $\mc{V}\colon \Omega \times X\to \mathbb C$  such that:
\begin{itemize}
\item for $\mathbb P$-a.e. $\omega \in \Omega$, we have that $\mc{V}(\om, \cdot)\in \B$;
\item \[
\esssup_{\om \in \Om} \lVert \mc{V}(\om, \cdot)\rVert_{\B}<\infty;
\]
\end{itemize}
Then, $\mc{S}'$ is a Banach space with respect to the norm
\[
\lVert \mc{V}\rVert_{\infty}:=\esssup_{\om \in \Om}\lVert \mc{V}(\om, \cdot)\rVert_{\B}.
\]
Furthermore, let $\mc{S}$ consist of all $\mc{V}\in \mc{S}'$ such that
 for $\mathbb P$-a.e. $\omega \in \Omega$, 
\[
\int_X \mc{V} (\om, \cdot)\, dm=0.
\]
Then,  $\mc{S}$ is a closed subspace of $\mc{S}'$ and therefore also a Banach space.

For $\theta \in \C^d$ and $\mc{W} \in \mc S$, set
\begin{equation}\label{defF}
F(\theta, \mc{W})(\om, \cdot)= \frac{\Lsot(\mc{W}(\sig^{-1}\om, \cdot) + v_{\sig^{-1}\om}^0(\cdot))}{\int \Lsot(\mc{W}(\sig^{-1}\om, \cdot) + v_{\sig^{-1}\om}^0(\cdot)) dm} - \mc{W}(\om,\cdot) - v_\om^0(\cdot).
\end{equation}

\begin{lemma}\label{lem:FwellDef}
There exist  $\ep, R>0$ such that
  $F \colon \mc{D} \to \mc{S}$ is a well-defined analytic map on $\mc{D}:=\{ \theta \in \C^d : |\theta|<\ep \} \times B_{\mc{S}}(0,R)$, where $B_{\mc{S}}(0,R)$ denotes the ball of radius $R$ in $\mc{S}$ centered at $0$.
\end{lemma}
\begin{proof}
Let $G \colon B_{\mathbb C^d}(0, 1) \times \mathcal S \to \mathcal S'$ and $H \colon B_{\mathbb C^d} (0, 1) \times \mc{S} \to L^\infty (\Om)$ be defined by~\eqref{GH}, where $B_{\C^d}(0,1)$ denotes the unit ball in $\C^d$.
 It follows from~\eqref{eq:boundedv} and Lemma~\ref{l49} that $G$ and $H$ are well-defined. Furthermore, by 
arguing as in~\cite[Lemma 5.1]{DFGTV2} we have that $G$ and $H$ are analytic. 

  Moreover,  since $H(0,0)(\om)=1$ for  $\om \in \Om$, we have that 
\begin{equation}\label{sca}
 \lvert H(\theta, \mc W)(\om)\rvert \ge 1-\lvert H(0,0)(\om)-H(\theta, \mc W)(\om)\rvert \ge 1-\lVert H(0,0)-H(\theta, \mc W)\rVert_{L^\infty},
\end{equation}
for \paeom. Hence, the continuity of $H$ implies that  $\lVert H(0,0)-H(\theta, \mc W)\rVert_{L^\infty} \le 1/2$
for all $(\theta, \mc W)$ in a neighborhood of $(0,0)\in \mathbb C^d \times \mc S$. We observe that it follows from~\eqref{sca} that in such neighborhood,
\[
 \essinf_{\om} \lvert H(\theta, \mc W)(\om)\rvert \ge 1/2.
\]
The above inequality together with~\eqref{eq:boundedv}  yields the desired conclusion.
\end{proof}
The proof of the following result follows closely the proof of~\cite[Lemma 3.5.]{DFGTV1}.
\begin{lemma}\label{thm:IFT}
Let $\mc{D}=\{ \theta \in \C^d : |\theta|<\ep \} \times B_{\mc{S}}(0,R)$ be as in Lemma~\ref{lem:FwellDef}. Then, by shrinking $\epsilon>0$ if necessary, we have that there exists $O\colon \{ \theta \in \C^d: |\theta|<\ep \} \to \mc S$ analytic in $\theta$ such that 
\begin{equation}
F(\theta, O(\theta))=0.
\end{equation}

\end{lemma}
\begin{proof}
We notice that  $F(0,0)=0$. Moreover, Proposition~\ref{difF} implies that
\[
 (D_{d+1}F(0,0) \mathcal X)(\om, \cdot)=\mathcal L_{\sigma^{-1} \om}(\mathcal X(\sigma^{-1}\om, \cdot))-\mathcal X(\om, \cdot) \quad \text{for $\om \in \Om$ and $\mathcal X\in \mathcal S$,}
\]
where $D_{d+1}F$ denotes the derivative of $F$ with respect to $\mc W$.
 We now prove that $D_{d+1} F(0,0)$ is bijective operator.

For injectivity, we have that if $D_{d+1}F(0, 0)\mathcal X=0$ for some nonzero $\mc{X}\in \mc{S}$, then $\mcl_\om \mc{X}_\om = \mc{X}_{\sig\om}$ for \paeom. Notice that $\mc{X}_\om \notin \langle v^0_\om \rangle$ because $\int \mc{X}_\om(\cdot) dm=0$ and  $ \mc{X}_\om\neq 0$. Hence, this yields a
contradiction with the one-dimensionality of the top Oseledets space of the cocycle $\mathcal L$, given by Lemma~\ref{lem:qc+1dim}. Therefore,  $D_{d+1}F(0,0)$ is injective.
To prove surjectivity, take $\mc{X}\in \mc{S}$ and
 let \begin{equation}\label{inverse} \tld{\mc{X}}(\om, \cdot):= - \sum_{j=0}^\infty  \mcl_{\sig^{-j}\om}^{(j)} \mc{X}(\sig^{-j}\om, \cdot). \end{equation} It follows from~\ref{cond:dec} that $\tld{\mc{X}}
 \in \mathcal S$ and it is easy to verify that
$D_{d+1} F(0,0)\tld{\mc{X}}=\mc{X}$.
Thus, $D_{d+1}F(0,0)$ is surjective.

Combining the previous arguments, we conclude that $D_{d+1}F(0,0)$ is bijective. The conclusion of the lemma  now follows directly from the  implicit complex analytic implicit function theorem in Banach spaces (see for instance the appendix in \cite{Val}).
\end{proof}

\subsection{On the top Lyapunov exponent for the twisted cocycle}
\label{LMB}
Let $\Lam(\theta)$ be the largest Lyapunov exponent associated to the twisted cocycle $\mcl^\theta$. Let $0<\ep<1$ and $O(\theta)$ be as in Lemma~\ref{thm:IFT}.
Let
\begin{equation}\label{eq:vomt}
v_\om^\theta(\cdot):= v_\om^0(\cdot) +O(\theta)(\om,\cdot).
\end{equation}
We notice that $\int v_\om^\theta(\cdot)\ dm =1$ and by Lemma~\ref{thm:IFT}, $\theta \mapsto v^\theta$ is analytic.
 Let us define
\begin{equation}\label{eq:hatLam}
 \hat\Lambda (\theta) :=  \int_\Om \log \Big|\int_X e^{\theta \cdot  g(\om, x)} v_\om^\theta(x) \,d m(x) \Big|\, d\bbp(\om),
\end{equation}
and
\begin{equation}\label{eq:int}
\lot :=  \int_X e^{\theta \cdot  g(\om, x)} v_\om^\theta(x) \,d m(x)
= \int_X \mcl_\om^\theta v_\om^\theta(x) \,d m(x),
\end{equation}
where the last identity follows from~\eqref{intprop}.

The proof of the following result is identical to the proof of~\cite[Lemma 3.8.]{DFGTV1}.
\begin{lemma}\label{lem:lowerBoundLam}
For every $\theta \in B_{\C^d}(0,\ep):= \{ \theta \in \C : |\theta|<\ep \}$, $ \hat\Lambda (\theta)\leq \Lambda (\theta)$.
\end{lemma}

The proof of the following result can be established by repeating the arguments in the proof of~\cite[Lemma 3.9.]{DFGTV1}.
\begin{lemma}\label{difflambda}
We have that $\hat\Lambda$ is differentiable on a neighborhood of 0, and for each $i\in \{1, \ldots, d\}$ we have that
\[
D_i\hat \Lambda (\theta)=
\Re \Bigg( \int_\Om \frac{ \overline{\lot}  ( \int_X g^i(\om, \cdot)e^{\theta \cdot g(\om, \cdot)}v_\om^\theta(\cdot)\, dm+\int_X  e^{\theta \cdot  g(\om, \cdot)}(D_i O(\theta))_\om (\cdot)\, dm )}{|\lot |^2}\, d\bbp(\om) \Bigg),
\]
where $\Re (z)$ denotes the real part of a complex number $z$ and $\overline{z}$ the complex conjugate of $z$. Here, $D_i$ denotes the derivative with respect to $\theta_i$, where $\theta=(\theta_1, \ldots, \theta_d)$.
\end{lemma}

\begin{lemma}\label{zero}
For $i\in \{1, \ldots, d\}$, 
 we have that $D_i \hat \Lambda(0)=0$.
\end{lemma}

\begin{proof}
Since $\lambda_\om^0=1$, it follows from the previous lemma  that 
\begin{equation}\label{6:20}
D_i \hat \Lambda(0)=\Re \Bigg( \int_\Om    \int_X \left(g^i(\om, \cdot)v_\om^0(\cdot)+ (D_i O(0))_\om (\cdot) \right) \, dm \, d\bbp(\om) \Bigg).
\end{equation}
On the other hand,  it follows from the implicit function theorem that
 \[
 D_i O(0)=-D_{d+1}F(0,0)^{-1}( D_iF(0,0)).
 \]
It was proved in Lemma~\ref{thm:IFT} that $D_{d+1}F(0,0) \colon \mathcal S \to \mathcal S$ is bijective. Thus, $D_{d+1}F(0,0)^{-1} \colon \mathcal S \to \mathcal S$ and therefore  $D_iO(0) \in \mathcal S$ which implies that
\begin{equation}\label{zerod}
\int_X D_iO(0)_\om \, dm=0 \text{ for \paeom.}
\end{equation}
 The conclusion of the lemma now
follows directly  from~\eqref{zeromean}, \eqref{6:20} and~\eqref{zerod}.
\end{proof}

The proofs of the following two results are identical to the proofs of~\cite[Theorem 3.12]{DFGTV1} and~\cite[Corollary 3.14.]{DFGTV1} respectively.
\begin{thm}[Quasi-compactness of twisted cocycles, $\theta$ near 0]\label{cor:quasicompactness}
Assume that the cocycle $\mc{R}=(\Om, \mathcal F, \bbp, \sig, \B, \mathcal L)$ is admissible. 
For $\theta \in \C^d$ sufficiently close to $0$, we have that the twisted cocycle $\mathcal L^\theta$ is quasi-compact.
Furthermore, for such $\theta$, the top Oseledets space of $\mathcal L^\theta$ is one-dimensional. That is, $\dim Y^\theta(\om)=1$ for \paeom.
\end{thm}

\begin{lemma}\label{cor:LamHatLam}
For  $\theta\in \C^d$ near 0, we have that $\Lam(\theta)=\hat\Lam(\theta)$.
In particular, $\Lam(\theta)$ is differentiable near $0$ and $D_i \Lam(0)=0$, for every $i\in \{1, \ldots, d\}$.
\end{lemma}

By arguing as in the proof of~\cite[Proposition 2.]{DH}, we have that there exists a positive semi-definite $d\times d$  matrix $\Sigma^2$  such that for $\mathbb P$-a.e. $\om \in \Om$, we have that 
\begin{equation}\label{VAR}
\Sigma^2=\lim_{n\to \infty}\frac 1 n \text{Cov}_\omega (S_n g(\om, \cdot)),
\end{equation}
where $\text{Cov}_\om$ denotes the e covariance with respect to the probability measure $\mu_\omega$. Moreover, the entries $\Sigma^2_{ij}$ of $\Sigma^2$ are given by
\begin{equation}\label{varelem}
\begin{split}
\Sig^2_{ij} &=\int_{\Omega \times X}g^i(\om, x)g^j (\om, x)\, d\mu(\om, x)+\sum_{n=1}^\infty \int_{\Omega \times X} g^i (\om, x)g^j (\tau^n (\om, x))\, d\mu(\om, x)\\
&\phantom{=}+\sum_{n=1}^\infty \int_{\Omega \times X} g^j (\om, x)g^i (\tau^n (\om, x))\, d\mu(\om, x).
\end{split}
\end{equation}
We also recall that $\Sig^2$ is positive definite if and only if $g$ does not satisfy that
\[
v\cdot g=r-r\circ \tau \quad \text{$\mu$-a.e.,}
\]
for all $v\in \R^d$, $v\neq 0$ and some $r\in L^2_{\mu}(\Omega \times X)$.

\begin{lemma}\label{lem:Lam''0}
We have that $\Lambda$ is of class $C^2$ on a neighborhood of $0$ and  $D^2 \Lam(0)=\Sig^2$, where $D^2\Lam (0)$ denotes the Hessian of $\Lambda$ in $0$.

\end{lemma}

\begin{proof}
By repeating the arguments in the proof of~\cite[Lemma 3.15.]{DFGTV1} one can show that $\Lambda$ is of class $C^2$ and that
\[
D_{ij} \Lam (\theta)=\Re  \bigg ( \int_\Om \bigg ( \frac{D_{ij}\lam_\om^\theta}{\lam_\om^\theta}-\frac{D_i \lam_\om^\theta D_j \lam_\om^\theta}{(\lam_\om^\theta)^2} \bigg ) \, d\mathbb P(\om) \bigg ),
\]
where $D_i\lam_\om^\theta$ denotes the derivative of $\theta \mapsto \lam_\om^\theta$ with respect to $\theta_i$ and $D_{ij}\lam_\om^\theta$ is the derivative of $\theta \mapsto D_j \lam_\om^\theta$ with respect to $\theta_i$. Moreover, using (\ref{eq:int}), the same arguments as in the proof of~\cite[Lemma 3.15.]{DFGTV1} yield that
\[
D_i \lam_\om^\theta=\int_X (g^i (\om, x)e^{\theta \cdot g(\om, x)}v_\om^\theta(x)+e^{\theta \cdot g(\om, x) }(D_i O(\theta))_\om (x))\, dm(x)
\]
and
\[
\begin{split}
D_{ij}\lam_\om^\theta &=\int_X(g^i (\om ,x)g^j (\om, x)e^{\theta \cdot g(\om, x)}v_\om^\theta (x) +g^j (\om, x)e^{\theta \cdot g(\om, x)}(D_i O(\theta))_\om (x))\, dm(x) \\
&\phantom{=}+\int_X (g^i (\om, x)e^{\theta \cdot g(\om, x)}(D_j O(\theta))_\om (x)+e^{\theta \cdot g(\om, x)} (D_{ij}O(\theta))_\om (x))\, dm(x).
\end{split}
\]
Since $D_{ij}O(0)\in S$ for $i, j\in \{1, \ldots, d\}$, we have that 
\begin{equation}\label{zerodd}
\int_X (D_{ij}O(0))_\om \, dm=0, \quad \text{for $\mathbb P$-a.e. $\om \in \Om$ and $i, j\in \{1, \ldots, d\}$.}
\end{equation}
From~\eqref{zeromean} and~\eqref{zerod}, we conclude that $D_i \lam_\om^\theta \rvert_{\theta=0}=0$ and 
\[
\begin{split}
D_{ij}\lam_\om^\theta \rvert_{\theta=0} &=\int_X g^i (\om ,x)g^j (\om, x)\, d\mu_\om (x) \\
&\phantom{=}+\int_X  (g^j (\om, x) (D_i O(0))_\om (x)+g^i(\om, x)(D_j O(0))_\om (x))\, dm(x).
\end{split}
\]
Hence,
\[
\begin{split}
D_{ij} \Lam (0) &= \Re \bigg (\int_{\Omega \times X} g^i (\om ,x)g^j (\om, x)\, d\mu(\om, x) +\int_{\Om} \int_X g^j (\om, x) (D_i O(0))_\om (x)\, dm(x)\, d\mathbb P(\om) \\
&\phantom{=}+\int_{\Om} \int_X g^i (\om, x) (D_j O(0))_\om (x)\, dm(x)\, d\mathbb P(\om) \bigg ).
\end{split}
\]
On the other hand, by the implicit function theorem, we have that
\[
D_i O(0)_\om=-(D_{d+1} F(0,0)^{-1}(D_i F(0,0)))_\om.
\]
Furthermore, \eqref{inverse} implies that
\[
(D_{d+1}F(0,0)^{-1}\mathcal W)_\om=-\sum_{n=0}^\infty \mathcal L_{\sigma^{-n}\om}^{(n)} (\mc W_{\sigma^{-n} \om}),
\]
for each $\mc W\in \mc S$. Hence, it follows from Proposition~\ref{difF} that 
 \[
D_i O(0)_\om=\sum_{n=1}^\infty \mathcal L_{\sigma^{-n}  \om}^{(n)} (g^i(\sigma^{-n} \om, \cdot) v_{\sigma^{-n} \om}^0(\cdot)).
\]
Consequently, since  $\sigma$ preserves $\mathbb P$, we have that
\[
\begin{split}
& \int_{\Om} \int g^j (\om, x) (D_i O(0))_\om (x)\, dm(x)\, d\mathbb P(\om) \\
 &=\sum_{n=1}^\infty \int_{\Om} \int_X  g^j (\om, x)\mathcal L_{\sigma^{-n}  \om}^{(n)} (g^i(\sigma^{-n} \om, \cdot) v_{\sigma^{-n} \om}^0)\, dm(x) \, d\mathbb P(\om) \\
&=\sum_{n=1}^\infty \int_{\Om} \int_X g^j (\om, T_{\sigma^{-n} \om}^{(n)}x)g^i (\sigma^{-n} \om, x)\, d \mu_{\sigma^{-n} \om} (x)\, d\mathbb P(\om) \\
&=\sum_{n=1}^\infty \int_{\Om} \int_X g^j (\sigma^n \om, T_{ \om}^{(n)}x)g^i ( \om, x)\, d \mu_{ \om} (x)\, d\mathbb P(\om) \\
&=\sum_{n=1}^\infty \int_{\Omega \times X} g^i(\om, x)g^j (\tau^n (\om, x))\, d\mu (\om, x).
\end{split}
\]
Thus, $D_{ij}\Lam (0)=\Sig_{ij}^2$ and the conclusion of the lemma follows. 
\end{proof}

\section{Limit theorems}
In this section we establish the main results of our paper. More precisely, we prove a number of limit laws for a broad classes of random piecewise dynamics and  for vector-valued observables. In particular, we prove the large deviations principle, central limit theorem and the local limit theorem, thus extending the main results in~\cite{DFGTV1} from scalar to vector-valued observables. In addition, we prove a number of additional limit laws that have not been discussed earlier. Namely, we establish the moderate deviations principle, concentration inequalities, self-normalized Berry-Esseen bounds as well as Edgeworth and large deviations (LD) expansions. 
\subsection{Choice of bases for top Oseledets spaces $Y_\omega ^\theta$ and $Y_\omega ^{*\theta}$}
\label{sec:choiceOsBases}

We recall that $Y_\omega^\theta$ and $Y_\omega^{*\theta}$ are top Oseledets subspaces for twisted and adjoint twisted cocycle, $\mcl^\theta$ and $\mcl^{\theta*}$, respectively.
 The Oseledets decomposition for these cocycles  can be written in the form
\begin{equation}\label{sod}
\BV=Y^\theta_\om \oplus H^\theta_\om  \quad \text{ and } \quad
\BV^* = Y^{*\,\theta}_\om \oplus H^{*\,\theta}_\om,
\end{equation}
where $H^\theta_\om=V^\theta(\omega)\oplus\bigoplus_{j=2}^{l_\theta} Y^\theta_j(\omega)$ is the equivariant complement to $Y^\theta_\om:= Y_1^\theta(\om)$, and $H^{*\,\theta}_\om$ is defined similarly.
Furthermore, Lemma~\ref{lem:AnnihilatorOsSplittings} shows that the following duality relations hold:
\begin{equation}\label{eq:DualityRel}
\begin{split}
\psi(y)&=0 \text{ whenever } y \in Y^\theta_\om \text{ and } \psi \in H^{*\,\theta}_\om,\quad \text{ and }\\
\phi(f) &=0 \text{ whenever } \phi  \in Y^{*\, \theta}_\om \text{ and } f \in  H^{\theta}_\om.
\end{split}
\end{equation}

Let us fix convenient choices for elements of the one-dimensional top Oseledets spaces $Y^\theta_\om$ and $Y^{*\,\theta}_\om$, for $\theta \in \C^d$ close to $0$.
Let  $v_\omega^\theta\in Y^\theta_\om$ be as in \eqref{eq:vomt}, so that $\int v_\om^\theta(\cdot)dm=1$. We recall that
\[
\mcl_\om^\theta v_\om^\theta=\lam_\om^\theta v_{\sigma \om}^\theta \quad \text{for $\mathbb P$-a.e. $\om \in \Om$,}
\]
where
\[
\lam_\om^\theta =\int e^{\theta \cdot g(\om, \cdot)}v_\om^\theta \, dm(x).
\]
Let us fix $\phi^\theta_\omega \in Y^{*\,\theta}_\om$ so that $\phi^\theta_\omega(v^\theta_\omega)=1$. We note that this selection is  possible and unique, because of~\eqref{eq:DualityRel}. Moreover, as in~\cite{DFGTV1} we easily conclude that
\[
(\mathcal{L}^{\theta}_\omega)^*\phi^\theta_{\sigma\omega}=\lambda^\theta_\omega \phi^\theta_{\omega}, \quad \text{for $\mathbb P$-a.e. $\om \in \Om$.}
\]

\subsection{Large deviations properties}

The proof of the following result is identical to the proof of~\cite[Lemma 4.2.]{DFGTV1}.
\begin{lemma}\label{L:growthExpSums}
Let $\theta\in \C^d$ be sufficiently close to 0, so that the results of Section~\ref{sec:choiceOsBases} apply.
Let $f\in \BV$ be such that $f\notin H_\om^\theta$, i.e. $\phi^\theta_\omega (f) \neq 0$.
Then,
\[
\lim_{n\to\infty}\frac{1}{n} \log \Big| \int e^{\theta \cdot  S_ng(\omega,\cdot )}f\ dm \Big| =  \Lam(\theta) \quad \text{for \paeom.}
\]
\end{lemma}

Next, suppose that $\Sig^2$ is positive definite and let $B\subset\mathbb R^d$ be a closed ball around the origin so that $D^2\Lambda(t)$ is positive definite for any $t\in B$ and set
\[
\Lambda^*(x)=\sup_{t\in B}\left(t\cdot x-\Lambda(t)\right).
\]
Observe that the existence of $B$ follows from Lemma~\ref{lem:Lam''0}.
By combining Lemma~\ref{L:growthExpSums} with Theorem~\ref{GEThm}, we obtain the following local large deviations principle.

\begin{thm}\label{LDthm} 
For $\mathbb P$-a.e. $\omega \in \Omega$,  we have:

(i) for any closed set $A\subset\mathbb R^d$,
\[
\limsup_{n\to\infty}\frac1{n}\log \mu_\om(\{S_n g(\om,\cdot)/n\in A \})\leq-\inf_{x\in A}\Lambda^*(x);
\]

(ii) there exists a closed ball $B_0$ around the origin (which does not depend on $\om$) so that for any open subset $A$ of $B_0$ we have
\[
\liminf_{n\to\infty}\frac1{n}\log \mu_\om(\{S_n g(\om,\cdot)/n\in A \})\geq-\inf_{x\in A}\Lambda^*(x).
\]
\end{thm}

\begin{rmk}
In the scalar case, for $\mathbb P$-a.e. $\om \in \Om$ and  for any sufficiently small $\ve>0$, we have (see~\cite[Lemma XIII.2.]{HennionHerve}) that 
\[
\lim_{n\to\infty}\frac1n \log \mu_\om (\{x: S_ng(\om,x)>n\ve\})=-\Lambda^*(\ve).
\]
The above conclusion was already obtained in~\cite[Theorem A]{DFGTV1}.

 In the multidimensional case we can apply \cite[Theorem 3.2.]{StLD}, and conclude that for any box $A$ around the origin with a sufficiently small diameter, 
\[
\lim_{n\to\infty}\frac1n \log \mu_\om(\{S_ng(\om,\cdot)/n\notin A \})=-\inf_{a\in\partial A}\Lambda^*(a).
\]
We also refer the reader to  \cite[Theorem 3.1.]{StLD} which, in particular, deals with the asymptotic behavior of probabilities of the form $\mu_\om(\{S_ng(\om,\cdot)/n\in C\})$, where $C$ is a cone with a  non-empty interior.
\end{rmk}

Next, we establish the following (optimal) global moderate deviations principle. 
Let $(a_n)_n$ be a sequence in $\R$ such that $\lim_{n\to\infty}\frac{a_n}{\sqrt n}=\infty$ and 
$\lim_{n\to\infty}\frac{a_n}n=0$. 

\begin{thm}\label{MDthm}
For $\mathbb P$-a.e.  $\omega \in \Omega$ and any $\theta\in\mathbb R^d$,  we have that
\[
\lim_{n\to\infty}\frac{1}{a_n^2/n}\log \mathbb  E[e^{\theta \cdot  S_ng (\om, \cdot)/c_n}]=\frac 1 2\theta^T\Sigma^2\theta,
\]
where $c_n=n/a_n$.
Consequently, when $\Sig^2$ is positive definite we have that:
\begin{enumerate}
\item[(i)] for any closed set $A\subset\mathbb R^d$,
\[
\limsup_{n\to\infty}\frac{1}{a_n^2/n}\log  \mu_\omega(\{S_n g(\om,\cdot)/a_n\in A\})\leq -\frac 1 2 \inf_{x\in A}x^T\Sigma^{-2} x;
\]
\item[(ii)] for any open set $A\subset\mathbb R^d$ we have
\[
\liminf_{n\to\infty}\frac{1}{a_n^2/n}\log  \mu_\omega(\{S_n g(\om,\cdot)/a_n\in A\})\geq-\frac 1 2 \inf_{x\in A}x^T\Sigma^{-2} x,
\]
where $\Sig^{-2}$ denotes the inverse of $\Sig^2$.
\end{enumerate}
\end{thm}

\begin{proof}
Let $\Pi_\omega(\theta)$ be an analytic branch of $\log \lambda_\omega^\theta$ around $0$ so that $\Pi_\omega(0)=0$ and $|\Pi_\omega(\theta)|\leq c$ for some $c>0$. Note that it is indeed possible to construct such functions $\Pi_\om$ in a deterministic neighborhood of $0$ since $\lam_\om^0=1$ and $\te\to\lam_\om^\te$ are analytic functions which are uniformly bounded around the origin. Set $\Pi_{\omega,n}(\theta)=\sum_{j=0}^{n-1}\Pi_{\sigma^j\om}(\theta)$.
Then $\nabla\Pi_\omega(0)=\nabla\lambda_\omega^{\theta}|_{\theta=0}=0$ (see the proof of Lemma \ref{lem:Lam''0})
and hence
\begin{equation}\label{D1 Pi}
\nabla\Pi_{\omega,n}(0)=0.
\end{equation}
 By applying $\mcl_\om^{\theta, (n)}$ to the identity  $v_\om^0=\phi_\om^\theta (v_\om^0)v_\om^\theta+ (v_\om^0-\phi_\om^\theta (v_\om^0)v_\om^\theta)$ and integrating with respect to $m$, we obtain that 
\begin{equation}\label{Basic relation}
\int_X e^{\theta \cdot S_ng(\omega,\cdot)}d\mu_\omega=\int_X \mcl_\om^{\theta, (n)} v_\om^0\, dm=
\phi_\omega^{\theta}(v_\omega^0)e^{\Pi_{\omega,n}(\theta)}+\int_X\mathcal L^{\theta,(n)}_\omega(v_\omega^{0}-\phi_\omega^{\theta}(v_\omega^0)v_\omega^\theta)dm.
\end{equation}
By Lemma \ref{lem:UnifExpDecayY2} the second term in the above right hand side is $O(r^n)$ uniformly in $\omega$ and $\theta$ (around the origin), for some $0<r<1$. Using  the Cauchy integral formula we get that
\begin{equation}\label{D2 Pi}
\left|D^2\Pi_{\omega,n}(0)-\text{Cov}_{\mu_\omega}(S_ng(\omega,\cdot))\right|\leq C,
\end{equation}
where $C$ is some constant which does not depend on $\omega$ and $n$. In the derivation of (\ref{D2 Pi}) we have also used that the function $\theta\to\phi_\om^\theta(v_\om^0)$ is analytic and uniformly bounded in $\om$, which can be proved as in~\cite[Appendix C]{DFGTV1}, using again the complex analytic implicit function theorem. 

Next, let $\theta\in\mathbb R^d$ and set $\theta_n=\theta/c_n$, where $c_n=n/a_n$ and $(a_n)_n$ is the sequence from the statement of the theorem. Then $\lim_{n\to\infty}c_n=\infty$ and $\lim_{n\to\infty}c_n^2/n=0$. Set $\Sigma^2_{\omega,n}=\text{Cov}_{\mu_\omega}(S_ng(\omega,\cdot))$. By (\ref{D2 Pi}), when $n$ is sufficiently large we can write 
\[
\Pi_{\omega,n}(\theta_n)=\frac12\theta_n^T\Sigma^2_{\omega,n}\theta_n+\mathcal O(|\theta_n|^2)+\mathcal O(n|\theta_n|^3).
\]
Therefore,
\[
\lim_{n\to\infty}\frac{c_n^2}{n}\Pi_{\omega,n}(\theta_n)=\frac12\theta^T\Sigma^2\theta.
\]
This together with (\ref{Basic relation}) implies that 
\[
\lim_{n\to\infty}\frac{c_n^2}{n}\log  \mathbb E[e^{\theta \cdot S_n g(\om, \cdot)/c_n}]=\lim_{n\to\infty}\frac{c_n^2}{n}\Pi_{\omega,n}(\theta_n)=\frac 12\theta^T\Sigma^2\theta.
\]
The upper and lower large deviations bounds follow now from the Gartner-Ellis Theorem (see~\cite[Theorem 2.3.6.]{DemZet}).
\end{proof}

Theorems \ref{LDthm} and \ref{MDthm} deal with the asymptotic behavior of probabilities of rare events on  an exponential scale. We will also obtain more explicit (but not tight) exponential upper bounds.

\begin{proposition}\label{ConcenProp}
There exist constants $c_1,c_2>0$ such  that for $\mathbb P$-a.e. $\om \in \Om$,  for any $\varepsilon>0$ and $n\in \N$ we have
\[
\mu_\omega (\{x\in X: |S_ng(\omega,x)|\geq \varepsilon n+c_1\} )\leq 2de^{-c_2 \varepsilon^2 n}.
\]
\end{proposition}

\begin{proof}
It is sufficient to establish the desired conclusion in the case when  $g$ is real-valued. Then by~\cite[(51)]{DH} there is a reverse martingale $\textbf{M}_n=X_1+...+X_n$ (which depends on $\omega$) with the following properties:
\begin{itemize}
\item there exists $c>0$ independent on $\om$ such that $\lVert X_i\rVert_{L^\infty (m)} \le c$;
\item there exists $C>0$ independent on $n$ and $\om$ such that
\begin{equation}\label{MartApp}
\sup_n \lVert S_n g(\omega,\cdot)-\textbf{M}_n(\cdot)\rVert_{L^\infty(m)}\leq C.
\end{equation}
\end{itemize}
The proof of the proposition is completed now using the Chernoff bounding method. More precisely, by 
 applying the Azuma-Hoeffding inequality with the martingale differences $Y_{k}=X_{n-k}$ we get that for any $\lambda>0$,
\[
\mathbb E_\om [e^{\lambda \textbf{M}_n}]\leq e^{\lambda^2c^2 n}.
\]
Therefore, by the Markov inequality we have that
\[
\mu_\omega (\{\textbf{M}_n\geq \varepsilon n \})=\mu_\om(\{e^{\lambda\textbf{M}_n}\geq e^{\lambda \varepsilon n}\})\leq 
e^{n(\lambda^2 c^2-\lambda\varepsilon)}.
\]
By taking $\lambda=\frac{\varepsilon}{2c^2}$, we obtain that $\mu_\om (\{\textbf{M}_n\geq \varepsilon n \})\leq e^{-\frac{\varepsilon^2}{4c^2}n}$.
Furthermore, by replacing $\textbf{M}_n$ with $-\textbf{M}_n$ we derive that
\[
\mu_\om (\{|\textbf{M}_n|\geq \varepsilon n \})\leq 2e^{-\frac{\varepsilon^2}{4c^2}n}.
\]
The proof of the proposition is completed using (\ref{MartApp}).
\end{proof}
\begin{rmk}
We remark that we can get upper bounds on the constants $c$ and $C$ appearing in the above proof, and so we can express $c_1$ and $c_2$ in terms of the parameters appearing in (V1)-(V8) and (C1)-(C5).
\end{rmk}

\subsection{Central limit theorem}
We need the following lemma. 
\begin{lemma}\label{lem:UnifExpDecayY2}
There exist $C>0$ and $0<r<1$ such that for every $\theta \in \mathbb{C}^d$ sufficiently close to 0, every $n\in \N$ and $\paeom$, we have
\begin{equation}
\Big| \int\mathcal L_\om^{\theta, (n)}(v_\om^0 -\phi_\om^{\theta}(v_\om^0) v_{\om}^{\theta})\, dm \Big|
\leq Cr^n \lvert \theta \rvert.
\end{equation}
\end{lemma}

\begin{proof}
The lemma now follows since the left hand side is $\mathcal O(r^n)$ uniformly in $\om$ and $\te$ (around the origin), it is analytic in $\theta$ and it vanishes at $\theta=0$ (and therefore, by the Cauchy integral formula its derivative is of order $\mathcal O(r^n)$ as well).
\end{proof}

\begin{thm}\label{CLT}
Assume the transfer operator cocycle $\mc{R}$ is admissible, and the observable $g$ satisfies conditions~\eqref{obs} and~\eqref{zeromean}. 
Assume also that the asymptotic covariance matrix $\Sig^2$ is positive definite. 
Then, for every bounded and continuous function $\phi \colon \R^d \to \R$ and \paeom, we have
\[
\lim_{n\to\infty}\int \phi \bigg{(}\frac{S_n g(\om, x)}{ \sqrt n}\bigg{)}\, d\mu_\om (x)=\int \phi \, d\mathcal N(0, \Sig^2).
\]
\end{thm}

\begin{proof}
It follows from Levy's continuity theorem that it is sufficient to prove that, for every $t\in \R^d$,
\[
 \lim_{n\to \infty}\int e^{it^T\frac{S_n g(\om, \cdot)}{\sqrt n}t} \, d\mu_\om=e^{-\frac{1}{2}t^T \Sig^2t} \quad \text{for \paeom,}
\]
where $t^T$ denotes the transpose of $t$.  Substituting $\te=t/\sqrt n$ in  (\ref{Basic relation}) and taking into account that $\lim_{\te\to 0}\phi_\om^\te(v_\om^0)=\phi_\om^0(v_\om^0)=1$, we conclude that it is sufficient to prove that 
\begin{equation}\label{uy}
\lim_{n \to \infty} \sum_{j=0}^{n-1} \log  \lambda_{\sigma^j \om}^{\frac{it}{\sqrt n}} = -\frac{1}{2} t^T \Sig^2 t, \quad \text{for \paeom.}
\end{equation}
We recall that $\lambda_\om^\theta=H(\theta, O(\theta))(\sigma \om)$, where $H$ is again given by~\eqref{GH}. We define $\tilde H$ on  a neighborhood of $0\in \C^d$ with values in $L^\infty(\Omega)$ by 
\[
\tilde H(\theta) (\om)=\log H(\theta, O(\theta)) (\om), \quad \om \in \Om.
\]
Observe that $\tilde H(0)(\om)=0$ for $\mathbb P$-a.e. $\om \in \Om$, and
that in the notations of the proof of Theorem \ref{MDthm} we have $\tilde H(\theta) (\om)=\Pi_{\sig^{-1}\om}(\te)$.
Therefore, as at the beginning of the proof of Theorem \ref{MDthm} we find that $\tilde H$ is analytic on  a neighborhood of $0$. Furthermore, by proceeding as in the proof of~\cite[Lemma 4.5.]{DFGTV1} we find that   
\[
D_i \tilde H(\theta)(\om)=\frac{1}{H(\theta, O(\theta)) (\om)}[D_i H(\theta, O(\theta)) (\om)+ (D_{d+1}H(\theta, O(\theta))D_i O(\theta))(\om)].  
\]
In particular, using Lemmas~\ref{l1} and~\ref{l2} we obtain that 
\[
D_i \tilde H(0)(\om)=\int g^i (\sigma^{-1}\om, \cdot )v_{\sigma^{-1}\om}^0 \, dm+\int (D_i O (0))_{\sigma^{-1} \om}\, dm.
\]
Thus, it follows from~\eqref{zeromean} and~\eqref{zerod} that $D_i \tilde H(0)(\om)=0$ for $i\in \{1, \ldots, d\}$ and for $\mathbb P$-a.e. $\om \in \Om$. 

Moreover, by  taking into account that $D_{d+1,d+1}H$ vanishes,  we have that  
\[
\begin{split}
& D_{ji}\tilde H(\theta)(\om)  \\
&=\frac{-E_i (\om) E_j(\om)}{[H(\theta, O(\theta)) (\om)]^2}  \\
&\phantom{=}+\frac{1}{H(\theta, O(\theta)) (\om)}[D_{ji}H(\theta, O(\theta)) (\om)+(D_{d+1, i}H(\theta, O(\theta)) D_j O(\theta))(\om)]\\
&\phantom{=}+\frac{1}{H(\theta, O(\theta)) (\om)} [(D_{j,d+1}H(\theta, O(\theta))D_i O(\theta))(\om)+(D_{d+1}H(\theta, O(\theta))D_{ji} O(\theta))(\om)],
\end{split}
\]
where 
\[
E_i(\om)=D_i H(\theta, O(\theta)) (\om)+ (D_{d+1}H(\theta, O(\theta))D_i O(\theta))(\om).
\]
By applying Lemma~\ref{laux} and using (\ref{zerodd}), we find that 
\begin{eqnarray*}
 D_{ji}\tilde H(0)(\om)=\int \Big(g^i(\sigma^{-1}\om, \cdot)g^j(\sigma^{-1}\om, \cdot)v_{\sigma^{-1}\om}^0+g^i (\sigma^{-1}\om, \cdot) (D_j O(0))_{\sigma^{-1} \om}+\\
g^j(\sigma^{-1}\om, \cdot)(D_i O(0))_{\sigma^{-1}\om}\Big)\, dm.
\end{eqnarray*}
Developing $\tilde{H}$  in a Taylor series around $0$, we have that
\[
\tilde{H}(\theta)(\om)= \log H(\theta, O(\theta))(\om) =\frac{1}{2} \theta^T  D^2 \tilde{H}(0)(\om) \theta+ R(\theta)(\om),
\]
where $R$ denotes the remainder. Therefore,
\[
 \log H \bigg{(}\frac{it}{\sqrt n}, O(\frac{it}{\sqrt n})\bigg{)}(\sigma^{ j+1} \om)=-\frac{1}{2n} t^T D^2 \tilde{H}(0)(\sigma^{j+1} \om)t+R(it/\sqrt n)(\sigma^{j+1} \om),
\]
which implies that
\begin{equation}\label{ii}
 \sum_{j=0}^{n-1} \log H \bigg{(}\frac{it}{\sqrt n}, O(\frac{it}{\sqrt n})\bigg{)}(\sigma^{ j+1} \om)=-\frac{1 }{2}\cdot \frac 1 n\sum_{j=0}^{n-1}t^T D^2 \tilde{H}(0)(\sigma^{j+1} \om) t+\sum_{j=0}^{n-1}R(it/\sqrt n)(\sigma^{j+1} \om).
 \end{equation}
By Birkhoff's ergodic theorem, we have that
\[
\begin{split}
\lim_{n\to \infty}\frac  1n \sum_{j=0}^{n-1}t^T D^2 \tilde{H}(0)(\sigma^{j+1} \om) t &=  \lim_{n\to \infty}\frac  1n \sum_{j=0}^{n-1}\sum_{k,l=1}^d t_k D_{kl}\tilde H(0) (\sigma^{j+1}\om ) t_l \\
&=\sum_{k,l=1}^d t_k D^2 \Lam (0)t_l \\
&=t^T \Sig^2 t,
\end{split}
\]
for $\mathbb P$-a.e. $\om \in \Om$, where we have used the penultimate equality in the Proof of Lemma \ref{lem:Lam''0}.
Furthermore, since $\tilde H(\te)$ are analytic in $\te$ and uniformly bounded in $\om$ we have that when $|t/\sqrt n|$ is sufficiently small then 
$|R(it/\sqrt n)(\om)|\leq C|t/\sqrt n|^3$, where $C>0$ is some constant which does not depend on $\om$ and $n$, and hence
\[
\lim_{n\to \infty}\sum_{j=0}^{n-1}R(it/\sqrt n)(\sigma^{j+1} \om) =0, \quad \text{for $\mathbb P$-a.e. $\om \in \Om$.}
\]
Thus, \eqref{ii} implies that
\[
\lim_{n\to \infty}\sum_{j=0}^{n-1} \log H \bigg{(}\frac{it}{\sqrt n}, O(\frac{it}{\sqrt n})\bigg{)}(\sigma^{ j+1} \om)=-\frac 1 2 t^T\Sig^2 t \quad \text{for $\mathbb P$-a.e. $\om \in \Om$,}
\]
and therefore~\eqref{uy} holds. This completes the proof of the theorem.
\end{proof}

\subsection{Berry-Esseen bounds}
In this subsection we restrict to the case when $d=1$, i.e. we consider real-valued observables. In this case $\Sigma^2$ is a nonnegative number and  in fact,
\[
\Sigma^2=\int_{\Omega \times X}g(\omega, x)^2\, d\mu(\omega, x)+2\sum_{n=1}^\infty \int_{\Omega \times X}g(\omega, x)g(\tau^n (\omega, x))\, d\mu(\omega, x).
\]
In this section we assume that $\Sig^2>0$ which means that $g$ is not an $L^2(\mu)$ coboundary with respect to the skew product $\tau$ (see~\cite[Proposition 3.]{DFGTV0}).
For $\omega \in \Omega$ and $n\in \N$, set
\[
\alpha_{\omega, n}:=\begin{cases}
\sum_{j=0}^{n-1}\tilde H''(0)(\sigma^{j+1}\omega) & \text{if $\sum_{j=0}^{n-1}\tilde H''(0)(\sigma^{j+1}\omega) \neq 0$;} \\
n\Sigma^2 & \text{if $\sum_{j=0}^{n-1}\tilde H''(0)(\sigma^{j+1}\omega)=0$,}
\end{cases}
\]
where $\tilde H$ is introduced in the previous subsection.  Then, 
\begin{equation}\label{8x}
\lim_{n\to \infty}\frac{\alpha_{\omega, n}}{n}=\Sigma^2, \quad \text{for $\mathbb P$-a.e. $\omega \in \Omega$.}
\end{equation}
Take now $\omega \in \Omega$ such that~\eqref{8x} holds. Set
\[
a_n:=\frac{\alpha_{\omega, n}}{n} \quad \text{and} \quad r_n=\sqrt{a_n},
\]
for $n\in \N$. Observe that $a_n$ and $r_n$ depend on $\om$ but in order to simplify the notation, we will not make this explicit. Taking $\te=tn^{-1/2}/r_n$ in (\ref{Basic relation}) we have that
\[
\begin{split}
\int_X e^{it\frac{S_ng(\omega, \cdot)}{r_n\sqrt n}}\, d\mu_\omega &=\phi_\omega^{\frac{it}{r_n\sqrt n}}(v_\omega^0)\prod_{j=0}^{n-1} \lambda_{\sigma^j \omega}^{\frac{it}{r_n \sqrt n}}\\
&\phantom{=}+\int_X\mathcal L_\omega^{\frac{it}{r_n \sqrt n}, (n)}(v_\omega^0-\phi_\omega^{\frac{it}{r_n\sqrt n}}(v_\omega^0)v_\omega^{\frac{it}{r_n \sqrt n}})\, dm.
\end{split}  
\]
Hence, 
\[
\begin{split}
\bigg{\lvert} \int_X e^{it\frac{S_ng(\omega, \cdot)}{r_n\sqrt n}}\, d\mu_\omega-e^{-\frac 1 2 t^2} \bigg{\rvert} &\le \bigg{\lvert}\phi_\omega^{\frac{it}{r_n\sqrt n}}(v_\omega^0)\prod_{j=0}^{n-1} \lambda_{\sigma^j \omega}^{\frac{it}{r_n \sqrt n}}-e^{-\frac 1 2 t^2} \bigg{\rvert} \\
&\phantom{\le}+\bigg{\lvert}\int_X\mathcal L_\omega^{\frac{it}{r_n \sqrt n}, (n)}(v_\omega^0-\phi_\omega^{\frac{it}{r_n\sqrt n}}(v_\omega^0)v_\omega^{\frac{it}{r_n \sqrt n}})\, dm \bigg{\rvert}.
\end{split}
\]
Observe that
\[
\begin{split}
\bigg{\lvert}\phi_\omega^{\frac{it}{r_n\sqrt n}}(v_\omega^0)\prod_{j=0}^{n-1} \lambda_{\sigma^j \omega}^{\frac{it}{r_n \sqrt n}}-e^{-\frac 1 2 t^2} \bigg{\rvert} &\le \lvert \phi_\omega^{\frac{it}{r_n\sqrt n}}(v_\omega^0)-1\rvert \cdot \bigg{\lvert}\prod_{j=0}^{n-1} \lambda_{\sigma^j \omega}^{\frac{it}{r_n \sqrt n}} \bigg{\rvert}\\
&\phantom{\le}+\bigg{\lvert}\prod_{j=0}^{n-1} \lambda_{\sigma^j \omega}^{\frac{it}{r_n \sqrt n}}-e^{-\frac 1 2 t^2}\bigg{\rvert} \\
&=:I_1+I_2.
\end{split}
\]
By~\cite[Lemma 4.6.]{DFGTV1}, we have that 
\[
\bigg{\lvert}\prod_{j=0}^{n-1} \lambda_{\sigma^j \omega}^{\frac{it}{r_n \sqrt n}} \bigg{\rvert} \le e^{-\frac{\Sigma^2}{8}r_n^{-2}t^2},
\]
for each $n\in \N$ sufficiently large  and $t$ such that $\lvert \frac{t}{r_n \sqrt n}\rvert$ is sufficiently small. Moreover, using the analyticity of the map $\theta \mapsto \phi^\theta$ (which as we already commented  can be obtained by repeating the arguments in~\cite[Appendix C]{DFGTV1}) and the fact\footnote{This is obtained by differentiating the identity $1=\phi_\om^\theta (v_\om^\theta)$ with respect to $\theta$ and evaluating at $\theta=0$.} that $\frac{d}{d\theta}\phi_\om^\theta \rvert_{\theta=0} (v_\om^0)=0$,  there exists $A>0$ (independent on $\om$ and $n$) such that 
\begin{equation}\label{Second order}
\lvert \phi_\omega^{\frac{it}{r_n\sqrt n}}(v_\omega^0)-1\rvert=\lvert \phi_\omega^{\frac{it}{r_n\sqrt n}}(v_\omega^0)-\phi_\omega^0(v_\omega^0)\rvert \le At^2r_n^{-2}n^{-1},
\end{equation}
whenever $\lvert \frac{t}{r_n \sqrt n}\rvert$ is sufficiently small. Consequently,  for $n$ sufficiently large and if $\lvert \frac{t}{r_n \sqrt n}\rvert$ is sufficiently small, 
\[
I_1 \le At^2r_n^{-2}n^{-1}e^{-\frac{\Sigma^2}{8}r_n^{-2}t^2}.
\]
On the other hand, we have that 
\[
\begin{split}
I_2 &=\bigg{\lvert} \prod_{j=0}^{n-1} \lambda_{\sigma^j \omega}^{\frac{it}{r_n \sqrt n}}-e^{-\frac 1 2 t^2}\bigg{\rvert} \\
&=\bigg{\lvert} e^{\sum_{j=0}^{n-1}\log \lambda_{\sigma^j \omega}^{\frac{it}{r_n \sqrt n}}}-e^{-\frac 1 2 t^2}\bigg{\rvert} \\
&=e^{-\frac 1 2 t^2}\bigg{\lvert} \exp \bigg{(}\sum_{j=0}^{n-1}\log \lambda_{\sigma^j \omega}^{\frac{it}{r_n \sqrt n}}+\frac 1 2 t^2 \bigg{)}-1\bigg{\rvert}.
\end{split}
\]
Observe that for $n$ sufficiently large, 
\[
\begin{split}
\sum_{j=0}^{n-1}\log \lambda_{\sigma^j \omega}^{\frac{it}{r_n \sqrt n}} &=\sum_{j=0}^{n-1}\tilde{H}(\frac{it}{r_n \sqrt n})(\sigma^{j+1}\omega)\\
&=\sum_{j=0}^{n-1} \bigg{(}\frac{-t^2 \tilde{H}''(0)(\sigma^{j+1} \omega)}{2nr_n^2}+R(\frac{it}{r_n\sqrt n})(\sigma^{j+1} \omega) \bigg{)}\\
&=-\frac{t^2}{2}+\sum_{j=0}^{n-1}R(\frac{it}{r_n\sqrt n})(\sigma^{j+1} \omega),
\end{split}
\]
and therefore
\[
\sum_{j=0}^{n-1}\log \lambda_{\sigma^j \omega}^{\frac{it}{r_n \sqrt n}}+\frac{t^2}{2}=\sum_{j=0}^{n-1}R(\frac{it}{r_n\sqrt n})(\sigma^{j+1} \omega).
\]
Using that $R(\frac{it}{r_n\sqrt n})=\frac{\tilde{H}'''(p_t)}{3!}(\frac{it}{r_n\sqrt n})^3$, for some $p_t$ between $0$ and $\frac{it}{r_n\sqrt n}$, we conclude that there exists $M>0$ such that
\[
\bigg{\lvert}\sum_{j=0}^{n-1}\log \lambda_{\sigma^j \omega}^{\frac{it}{r_n \sqrt n}}+\frac{t^2}{2} \bigg{\rvert} \le nM\lvert \frac{it}{r_n \sqrt n}\rvert^3 =\frac{M\lvert t\rvert^3}{r_n^3 \sqrt n}.
\]
Since $\lvert e^z-1\rvert \le 2\lvert z\rvert$ whenever $\lvert z\rvert$ is sufficiently small, we conclude that
\[
I_2\le 2Me^{-\frac{t^2}{2}}\lvert t\rvert^3r_n^{-3}n^{-1/2}.
\]
Observe that Lemma~\ref{lem:UnifExpDecayY2} implies that 
\[
\bigg{\lvert}\int_X\mathcal L_\omega^{\frac{it}{r_n \sqrt n}, (n)}(v_\omega^0-\phi_\omega^{\frac{it}{r_n\sqrt n}}(v_\omega^0)v_\omega^{\frac{it}{r_n \sqrt n}})\, dm \bigg{\rvert} \le C\frac{r^n \lvert t\rvert}{r_n \sqrt n},
\]
for some $C>0$ and whenever $\lvert \frac{t}{r_n \sqrt n}\rvert$ is sufficiently small.

Let $F_n \colon \R \to \R$ be a distribution function of $\frac{S_ng(\omega, \cdot)}{r_n \sqrt n}=\frac{S_n g(\om, \cdot)}{\sqrt {\alpha_{\om, n}}}$. Furthermore, let $F\colon \R \to \R$ be a distribution function of $\mathcal N(0, 1)$. Then, it follows from Berry-Esseen inequality that 
\begin{equation}\label{Esseen}
\sup_{x\in \R}\lvert F_n(x)-F(x)\rvert \le \frac{2}{\pi}\int_0^T\bigg{\lvert}\frac{\mu_\omega(e^{\frac{itS_ng(\omega, \cdot)}{r_n \sqrt n}})-e^{-\frac 1 2 t^2}}{t}\bigg{\rvert}\, dt +\frac{24}{\pi T}\sup_{x\in \R}\lvert F'(x)\rvert,
\end{equation}
for any $T>0$. It follows from the estimates we established that there exists $\rho>0$ such that
\[
\begin{split}
\int_0^{\rho r_n \sqrt n}\bigg{\lvert}\frac{\mu_\omega(e^{\frac{itS_ng(\omega, \cdot)}{r_n \sqrt n}})-e^{-\frac 1 2 t^2}}{t}\bigg{\rvert}\, dt  &\le Ar_n^{-2}n^{-1}\int_0^\infty te^{-\frac{\Sigma^2}{8}r_n^{-2}t^2}\, dt \\
&\phantom{\le}+2Mr_n^{-3}n^{-\frac 1 2}\int_0^\infty t^2 e^{-\frac{t^2}{2}}\, dt \\
&\phantom{\le}+C\rho r^n,
\end{split}
\]
for sufficiently large $n$. Since
\[
\sup_n \int_0^\infty te^{-\frac{\Sigma^2}{8}r_n^{-2}t^2}\, dt <\infty \quad \text{and} \quad \int_0^\infty t^2 e^{-\frac{t^2}{2}}\, dt <\infty, 
\]
we conclude that
\begin{equation}\label{BE1}
\sup_{x\in \R}\lvert F_n(x)-F(x)\rvert \le R(\omega)n^{-\frac 1 2}, 
\end{equation}
for some random variable $R$. 

Next, notice that in the notations of the proof of Theorem \ref{MDthm} we have 
\[
\al_{\om,n}=\Pi_{\om,n}''(0).
\]
Set $\sig_{\om,n}^2=\var_{\mu_\om}\big(S_ng(\om,\cdot)\big)$. Then by (\ref{D2 Pi}) we have
\[
|\sig_{\om,n}^2-\al_{\om,n}|\leq C,
\]
where $C$ is some constant which does not depend on $n$. Since $\al^{-\frac12}-\sig^{-\frac12}=\frac{\sig-\al}{\sqrt{\al\sig}(\sqrt \alpha+\sqrt \sigma)}$ for any nonzero $\al$ and $\sig$,  taking into account (\ref{obs}) we have 
\[
\left|S_ng(\om,\cdot)/\sqrt{\al_{\om,n}}-S_ng(\om,\cdot)/\sig_{\om,n}\right|\leq C_1n^{-\frac12}
\]
for some constant $C_1$ which does not depend on $n$. By applying~\cite[Lemma 3.3]{HK-BE} with $a=\infty$, we conclude from (\ref{BE1}) that the following self-normalized version of the Berry-Esseen theorem holds true:
\begin{equation}\label{BE2}
\sup_{x\in \R}\lvert \bar{F}_n(x)-F(x)\rvert \le R_1(\omega)n^{-\frac 1 2}
\end{equation}
for some random variable $R_1$, where $\bar{F}_n$ is a distribution function of $\frac{S_n g(\om, \cdot)}{\sigma_{\om, n}}$. 
\begin{rmk}
We stress that analogous result (using different techniques) for random expanding dynamics was obtained in~\cite[Theorem 7.1.1.]{HK}.
In Theorem~\ref{EdgeThm} we will give a somewhat different proof of~\eqref{BE2}, as well as prove  certain Edgeworth expansions of order one.
\end{rmk}

\subsection{Local limit theorem}

\begin{thm}\label{LLTthm}
Suppose that $\Sigma^2$ is positive definite and that for any compact set $J\subset\mathbb R^d\setminus\{0\}$ there exist $\rho \in (0,1)$ and a random variable $C\colon \Om \to (0, \infty)$ such that 
\begin{equation}\label{Large t's}
\lVert \mcl_\om^{it, (n)}\rVert \le C(\om) \rho^n, \quad \text{for $\mathbb P$-a.e. $\om \in \Om$, $t\in J$ and $n\in \N$.}
\end{equation}
Then, for $\mathbb P$-a.e. $\om \in \Om$ we have that 
\[
\lim_{n\to \infty}\sup_{s\in \R^d} \bigg{\lvert} \lvert \Sig \rvert n^{d/2}\mu_\om (s+S_n g(\om, \cdot)\in J)-\frac{1}{(2\pi)^{d/2}}e^{-\frac{1}{2n}s^T\Sigma^{-2}s}\lvert J\rvert \bigg{\rvert}=0, 
\]
where $\lvert \Sig \rvert=\sqrt{\det \Sig^2}$, $\Sig^{-2}$ is the inverse of $\Sig^2$ and $\lvert J\rvert$ denotes the volume of $J$. 
\end{thm}

\begin{proof}
The proof is analogous to the proof of~\cite[Theorem C.]{DFGTV1}.
 Using the density argument (analogous to that in~\cite{Morita}), it is sufficient to show that
 \begin{equation}
 \label{66a}
 \sup_{s\in \R^d} \bigg{\lvert} \lvert \Sig \rvert  n^{d/2}\int h(s+S_ng(\om, \cdot))\, d\mu_\om-\frac{1}{(2\pi)^{d/2}}e^{-\frac{1}{2n}s^T\Sigma^{-2}s}\int_{\R^d} h(u)\, du\bigg{\rvert} \to 0,
 \end{equation}
when $n\to \infty$ for every $h\in L^1(\R^d)$ whose Fourier transform $\hat{h}$ has compact support. By using the inversion formula
\[
 h(x)=\frac{1}{(2\pi)^d}\int_{\R^d} \hat h(t)e^{it\cdot x}\, dt,
\]
and Fubini's theorem we have that 
\[
 \begin{split}
\lvert  \Sig  \rvert n^{d/2}\int h(s+S_ng(\om,\cdot))\, d\mu_\om &=\frac{\lvert \Sig \rvert  n^{d/2} }{(2\pi)^d} \int \int_{\R^d}\hat h(t)e^{it \cdot (s+S_n g(\om, \cdot))}\, dt \, d\mu_\om \\
  &=\frac{\lvert \Sig \rvert  n^{d/2} }{(2\pi)^d} \int_{\R^d} e^{it\cdot s}\hat h(t)\int e^{it \cdot S_n g(\om, \cdot)}\, d\mu_\om \, dt \\
  &=\frac{\lvert \Sig \rvert  n^{d/2} }{(2\pi)^d} \int_{\R^d} e^{it\cdot s}\hat h(t)\int e^{it \cdot S_n g(\om, \cdot)}v_\om^0\, dm \, dt \\
  &=\frac{\lvert \Sig \rvert  n^{d/2} }{(2\pi)^d} \int_{\R^d} e^{it\cdot s}\hat h(t)\int \mathcal L_{\om}^{it, (n)}v_\om^0\, dm \, dt \\
  &=\frac{\lvert \Sigma \rvert}{(2\pi)^d}\int_{\R^d} e^{\frac{it \cdot s}{\sqrt n}}\hat h(\frac{t}{\sqrt n})\int \mathcal L_{\om}^{\frac{it}{\sqrt n}, (n)}v_\om^0\, dm \, dt.
 \end{split}
\]
Recalling that the Fourier transform of  $f(x)=e^{-\frac{1}{2}x^T\Sig^2 x}$ is given by $\hat f(t)=\frac{(2\pi)^{d/2}}{\lvert \Sigma \rvert}e^{- \frac 1 2 t^T\Sig^{-2}t}$, we have that 
\begin{eqnarray*}
\frac{1}{(2\pi)^{d/2}}e^{-\frac{1}{2n}s^T\Sig^{-2}s}\int_{\R^d} h(u)\, du&=&\frac{\hat h(0)}{(2\pi)^{d/2}}e^{-\frac{1}{2n}s^T\Sig^{-2}s}\\
&=&\frac{\hat h(0) \lvert \Sigma \rvert}{(2\pi)^d} \hat f(-s/\sqrt n) \\
&=&\frac{\hat h(0)\lvert \Sigma \rvert}{(2\pi)^d} \int_{\R^d} e^{\frac{it\cdot s}{\sqrt n}} \cdot e^{-\frac{1}{2}t^T\Sig^2 t}\, dt.
 \end{eqnarray*}
Therefore, in order to complete the proof of the theorem we need to show that 
\[
\sup_{s\in \R^d} \bigg{\lvert}\frac{\lvert \Sigma \rvert}{(2\pi)^d}\int_{\R^d} e^{\frac{it \cdot s}{\sqrt n}}\hat h(\frac{t}{\sqrt n})\int \mathcal L_{\om}^{\frac{it}{\sqrt n}, (n)}v_\om^0\, dm \, dt-\frac{\hat h(0)\lvert \Sigma \rvert}{(2\pi)^d} \int_{\R^d} e^{\frac{it\cdot s}{\sqrt n}} \cdot e^{-\frac{1}{2}t^T\Sig^2 t}\, dt
\bigg{\rvert} \to 0,
\]
when $n\to \infty$ for $\mathbb P$-a.e. $\om \in \Om$. Choose $\delta >0$ such that the support of $\hat h$ is contained in $\{t\in \R^d: \lvert t\rvert \le \delta \}$. Then, for any $\tilde \delta \in (0, \delta)$, we have that 
\begin{align*}
 &\frac{\lvert \Sigma \rvert}{(2\pi)^d}\int_{\R} e^{\frac{it\cdot s}{\sqrt n}}\hat h(\frac{t}{\sqrt n})\int \mathcal L_{\om}^{\frac{it}{\sqrt n}, (n)}v_\om^0\, dm \, dt -
 \frac{\hat h(0)\lvert \Sigma \rvert}{(2\pi)^d} \int_{\R} e^{\frac{it\cdot s}{\sqrt n}} \cdot e^{-\frac{1}{2}t^T\Sig^2 t}\, dt \displaybreak[0] \\
 &=\frac{\lvert \Sigma \rvert}{(2\pi)^d} \int_{\lvert t\rvert < \tilde \delta \sqrt n} e^{\frac{it\cdot s}{\sqrt n}} \Big(\hat h(\frac{t}{\sqrt n})\prod_{j=0}^{n-1}\lambda_{\sigma^j \om}^{\frac{it}{\sqrt n}}-\hat h(0)e^{-\frac{1}{2}t^T\Sig^2 t} \Big)\, dt \displaybreak[0] \\
&\phantom{=}+\frac{\lvert \Sigma \rvert}{(2\pi)^d}\int_{\lvert t\rvert < \tilde \delta \sqrt n}e^{\frac{it\cdot s}{\sqrt n}}\hat h(\frac{t}{\sqrt n})
\int
\prod_{j=0}^{n-1}\lambda_{\sigma^j \om}^{\frac{it}{\sqrt n}} \Big( \phi_\om^{\frac{it}{\sqrt n}}( v_\om^0 ) v_{\sig^{n}\om}^{\frac{it}{\sqrt n}}-1 \Big)\,dm \, dt \displaybreak[0] \\
&\phantom{=}+\frac{\lvert\Sig \rvert n^{d/2}}{(2\pi)^d}\int_{\lvert t\rvert <\tilde \delta}e^{it\cdot s}\hat h(t)\int  \mathcal L_{\om}^{it, (n)} (v_\om^0 - \phi_\om^{it}( v_\om^0 ) v_{\om}^{it}) \, dm\, dt \displaybreak[0] \\
&\phantom{=}+\frac{\lvert \Sig \rvert n^{d/2}}{(2\pi)^d}\int_{\tilde \delta \le \lvert t\rvert < \delta}e^{it\cdot s}\hat h(t)\int \mathcal L_\om^{it, (n)}v_\om^0\, dm\, dt \displaybreak[0] \\
&\phantom{=}-\frac{\lvert \Sigma \rvert}{(2\pi)^d}\hat h(0) \int_{\lvert t\rvert \ge \tilde \delta \sqrt n}e^{\frac{it\cdot s}{\sqrt n}} \cdot e^{-\frac 1 2 t^T\Sig^2 t}\, dt=: (I)+(II)+(III)+(IV)+(V).
\end{align*}
One can now proceed as in the proof of~\cite[Theorem C.]{DFGTV1} and show that each of the terms $(I)$--$(V)$ converges to zero as $n\to \infty$. For the convenience of the reader, we give here  complete arguments for terms $(I)$ (which is most involved) and $(IV)$ (since this is the only part of the proof that requires~\eqref{Large t's}).

\paragraph{Control of (I).}
We claim that for \paeom,
\[
\lim_{n\to \infty}
\sup_{s\in \R^d} \bigg{\lvert}\int_{\lvert t\rvert < \tilde \delta \sqrt n}e^{\frac{it\cdot s}{\sqrt n}}
 \Big(\hat h(\frac{t}{\sqrt n})\prod_{j=0}^{n-1} \lambda_{\sigma^j \om}^{\frac{it}{\sqrt n}}-\hat h(0)e^{-\frac 1 2 t^T\Sig^2 t}\Big)\, dt \bigg{\rvert} = 0.
 \]
Observe that 
 \[
\begin{split}
&  \sup_{s\in \R^d} \bigg{\lvert}\int_{\lvert t\rvert < \tilde \delta \sqrt n}e^{\frac{it\cdot s}{\sqrt n}}
 \Big(\hat h(\frac{t}{\sqrt n})\prod_{j=0}^{n-1} \lambda_{\sigma^j \om}^{\frac{it}{\sqrt n}}-\hat h(0)e^{-\frac 1 2 t^T\Sig^2 t} \Big)\, dt \bigg{\rvert} \\
& \le \int_{\lvert t\rvert < \tilde \delta \sqrt n}\bigg{\lvert} \hat h(\frac{t}{\sqrt n})\prod_{j=0}^{n-1} \lambda_{\sigma^j \om}^{\frac{it}{\sqrt n}}-\hat h(0)e^{-\frac 1 2 t^T\Sig^2 t}
 \bigg{\lvert} \, dt.
\end{split}
 \]
It follows from the continuity of $\hat h$ and~\eqref{uy} that for \paeom \ and every $t$,
\begin{equation}
\label{intermediate}
 \hat h(\frac{t}{\sqrt n})\prod_{j=0}^{n-1} \lambda_{\sigma^j \om}^{\frac{it}{\sqrt n}}-\hat h(0)e^{-\frac 1 2 t^T\Sig^2 t}\to 0, \quad \text{when $n\to \infty$.}
\end{equation}
The desired conclusion will  follow from the dominated convergence theorem once we establish the following lemma.

\begin{lemma}\label{lem:boundProdLam}
 For $\tilde \delta >0$ sufficiently small, there exists $n_0\in \N$ such that for all $n\ge n_0$ and $t$ such that $\lvert t\rvert < \tilde \delta \sqrt n$,
 \[
  \bigg{\lvert} \prod_{j=0}^{n-1} \lambda_{\sigma^j \om}^{\frac{it}{\sqrt n}} \bigg{\rvert} \le e^{-\frac 1 8 t^T\Sig^2 t}.
 \]
\end{lemma}

\begin{proof}[Proof of the lemma]
We will use the same notation as in the proof of Theorem~\ref{CLT}.  We have that
 \[
  \bigg{\lvert} \prod_{j=0}^{n-1} \lambda_{\sigma^j \om}^{\frac{it}{\sqrt n}} \bigg{\rvert}=e^{-\frac{1}{2n} \Re ( \sum_{j=0}^{n-1} t^T D^2 \tilde{H}''(0)(\sigma^{j+1} \om) t)}\cdot
  e^{\Re (\sum_{j=0}^{n-1} R(it/\sqrt n)(\sigma^{j+1} \om))}.
 \]
In the proof of Theorem~\ref{CLT} we have showed that \[\frac{1}{n} \sum_{j=0}^{n-1}  D^2 \tilde{H}(0)(\sigma^{j+1} \om) \to \Sigma^2  \quad \text{for \paeom.} \]
Therefore, for $\mathbb P$-a.e. $\om \in \Om$ there exists $n_0=n_0(\om) \in \N$ such that
\[
e^{-\frac{1}{2n} \Re ( \sum_{j=0}^{n-1} t^T D^2 \tilde{H}''(0)(\sigma^{j+1} \om) t)} \le e^{-\frac 1 4 t^T\Sig^2 t}, \quad \text{for $n\ge n_0$ and $t\in \R^d$.}
\]
Finally, recall that $|R(it/\sqrt n)(\om)|\leq C|t/\sqrt n|^3$, where $C>0$ is some constant which does not depend on $\om$ and $n$ when $|t/\sqrt n|$ is small enough. Therefore, if $|t|\leq\sqrt n\tilde\del$ and $\tilde\del$ is small enough we have
 we have 
\[
 e^{\Re (\sum_{j=0}^{n-1} R(it/\sqrt n)(\sigma^{j+1} \om))} \le e^{C|t|^3n^{-\frac12}} \leq e^{-\frac 1 8 t^T\Sig^2 t}.
\]
Here we have used that $|t|^3n^{-1/2}\leq\tilde\del |t|^2$ and that $t^T\Sig^2 t\geq a|t|^2$ for some $a>0$ and all $t\in\bbR^d$. 
The conclusion of the lemma follows directly from the last two estimates. 
\end{proof}
\paragraph{Control of (IV).}
By~\eqref{Large t's},
\[
 \sup_{s\in \R^d}\frac{\lvert \Sig \rvert n^{d/2}}{(2\pi)^d} \bigg{\lvert}\int_{\tilde \delta \le \lvert t\rvert \le \delta} e^{it \cdot s}\hat h(t)\int \mathcal L_{\om}^{it, (n)}v_\om^0\, dm \, dt
 \bigg{\rvert} \le CV_{\delta, \tilde \delta}\frac{ \lvert \Sig \rvert n^{d/2}}{(2\pi)^d}\lVert \hat h\rVert_{L^\infty} \cdot \rho^n \cdot \lVert v^0\rVert_\infty  \to 0,
\]
when $n\to \infty$ by \eqref{eq:boundedv} and the fact that $\hat h$ is continuous. Here $V_{\delta, \delta'}$ denotes the volume of $\{t\in \R^d: \tilde \delta \le \lvert t\rvert \le \delta\}$.

\end{proof}

Let us now discuss  conditions under which~\eqref{Large t's} holds.

\begin{lemma}\label{Per0}
Assume that: 
\begin{enumerate}
\item $\mathcal F$ is a Borel $\sigma$-algebra on $\Omega$;
\item $\sigma$ has a periodic point $\om_0$ (whose period is denoted by $n_0$), and $\sigma$ is continuous at each point that belongs to the orbit of $\om_0$;
\item $\mathbb P(U)>0$ for any open set $U$ that intersects the orbit  of $\om_0$;
\item for any compact set $J\subset\bbR^d$,  the family of  maps $\om\to\cL_\om^{it},\,t\in J$ is uniformly continuous at the orbit points of $\om_0$;
\item for any  $t\not=0$ the spectral radius of $\cL_{\om_0}^{it,(n_0)}$ is smaller than $1$;
\item  for any compact set $J\subset\bbR^d$ there exists a constant $B(J)>0$ such that 
\begin{equation}\label{Bound J}
\sup_{t\in J}\sup_{n\geq 1}\|\cL_\om^{it,(n)}\|\leq B(J).
\end{equation}
\end{enumerate}
Then, for any compact $J\subset\bbR^d\setminus\{0\}$ there exists a random variable $C\colon \Omega \to (0, \infty)$ and  a constant $d=d(J)>0$ such that for $\mathbb P$-a.e. $\om \in \Om$ and  for any $n\geq 1$, we have that 
\[
\sup_{t\in J}\|\cL_\om^{it,(n)}\|\leq C(\om ) e^{-nd}.
\]
\end{lemma}
The proof of Lemma~\ref{Per0} is identical to the proof of ~\cite[Lemma 2.10.4]{HK}.
We also refer the readers to the arguments in proof of Lemma \ref{Per}.
Condition (\ref{Bound J}) is satisfied for the distance expanding maps considered in~\cite[Chapter 5]{HK} (assuming they are non-singular). 
Indeed, the proof of the Lasota-Yorke inequality (see~\cite[Lemma 5.6.1]{HK}) proceeds similarly for vectors $z\in\bbC^d$ instead of complex numbers. Therefore, there exists a constant $C>0$ so that $P$-almost surely, for any $t\in\bbR^d$ and $n\geq 1$ we have
\[
\|\cL_\om^{it,(n)}\|\leq C(1+|t|)\sup|\cL_\om^{(n)}\textbf{1}|
\]
where $\textbf{1}$ is the function which takes the constant value $1$. Note that in the circumstances of \cite{HK},\, $\var(\cdot)=v_\al(\cdot)$ is the H\"older constant corresponding to some exponent $\al\in(0,1]$. In particular $\mathcal B$ contains only H\"older continuous functions and the norm $\|\cdot\|_{\mathcal B}$ is equivalent to the norm $\|g\|_\al=v_\al(g)+\sup|g|$. Therefore, by \ref{cond:dec} for $P$-almost any $\om$ we have 
\[
\sup|\cL_\om^{(n)}\textbf{1}|=\|\cL_\om^{(n)}\textbf{1}\|_{L^\infty}\leq C
\]
for some $C$ which does not depend on $\om$ and $n$, and hence (\ref{Bound J}) holds true.

\subsection{Edgeworth and LD expansions}
Let us restrict ourselves again to the scalar case $d=1$.
Our main result here is the following Edgeworth expansion of order $1$.
\begin{thm}\label{EdgeThm}
Suppose that $\Sigma^2>0$. 

(i) The following self normalized version of the Berry-Esseen theorem holds true:
\begin{equation}\label{SelfNor BE}
\sup_{t\in\bbR}\left|\mu_\omega(\{S_ng(\omega,\cdot)\leq t\sigma_n\})-\Phi(t)\right|\leq R_\om n^{-\frac12},
\end{equation}
for some random variable  $R_\om$, where $\Phi(t)$ is the standard normal distribution function and $\sigma_n^2=\sigma_{\omega,n}^2=\text{Var}_{\mu_\omega}(S_ng(\omega,\cdot))$.

(ii) Assume, in addition, that for any compact set $J\subset\mathbb R\setminus\{0\}$ we have
\begin{equation}\label{Large t's2}
\lim_{n\to\infty}n^{1/2}\left|\int_J\int_X e^{\frac{it}{\sqrt{n}}  S_n g(\omega,x)}d\mu_\omega(x)dt\right|=0,\,P\text{-a.s.}
\end{equation}
Let $A_{\omega,n}$ be a function whose derivative's Fourier transform is $e^{-\frac 12 t^2}(1+\cP_{\om,n}(t))$, where
\[
\cP_{\om,n}(t)=-\frac12\Pi_{\omega,n}''(0)\left(\frac{t}{\sigma_n}\right)^2+\frac12t^2-\frac{i}6\Pi_{\omega,n}'''(0)\left(\frac{t}{\sigma_n}\right)^3.
\]
Then,
\[
\lim_{n\to\infty}\sqrt n\sup_{t\in\mathbb R}\left|\mu_\omega(\{S_ng(\omega,\cdot)\leq t\sigma_n\})-A_{\omega,n}(t)\right|=0.
\]
\end{thm}
Before proving Theorem~\ref{EdgeThm} let us introduce some additional notation and make some observations. 
It is clear that $A'_{\om,n}$ has the form $A'_{\om,n}(t)=Q_{\om,n}(t)e^{-\frac12 t^2}$ where $Q_{\om,n}(t)$ is a polynomial of degree $3$. In fact, if we set  $a_{n,\om}=\frac12\big(1-\Pi_{\om,n}''(0)/\sig_n^2\big)$ and $b_{\om,n}=\frac16\Pi_{\om,n}'''(0)/\sig_n^3$, we have that
\begin{equation}\label{Q om n}
\sqrt{2\pi}Q_{\om,n}(t)=1+a_{\om,n}+3b_{\om,n}t-a_{\om,n}t^2-b_{\om,n}t^3.
\end{equation}
By (\ref{D2 Pi}) we have $a_{\om,n}=\mathcal O(1/n)$, while $b_{\om,n}=\mathcal O(1/\sqrt n)$ (since $|\Pi_{\om,n}'''(0)|\leq cn$). Set $\varphi(t)=\frac1{\sqrt{2\pi}}e^{-\frac 12 t^2}$ and $u_{\om,n}=\frac{\Pi_{\om,n}^{(3)}(0)}{\sig_{n}^2}$, which converges to $\Sig^{-2}\int \Pi_\om^{(3)}(0)dP(\om)$ as $n\to\infty$. Using the above formula of $Q_{\om,n}$ together with $a_{\om,n}=\mathcal O(1/n)$ we conclude that
\[
\lim_{n\to\infty}\sqrt n\sup_{t\in\mathbb R}\left|\mu_\omega(\{S_ng(\omega,\cdot)\leq t\sigma_n\})-\Phi(t)-u_{\om,n}\sig_n^{-1}(t^2-1)\varphi(t)\right|=0.
\] 
\begin{rmk}
We remark that in the deterministic case (i.e. when $\Om$ is a singleton) we have $a_{\om,n}=0$ and $\Pi_{\om,n}'''(0)=n\ka_3$ for some $\ka_3$ which does not depend on $n$. Therefore, $u_{\om,n}=\ka_3\Sig^{-2}$ and we recover the  order one deterministic Edgeworth expansion that was established in \cite{FL}.
It seems unlikely that we can get the same results in the random case since this would imply that 
\[
\left|\Pi_{\om,n}^{(3)}(0)/n-\int \Pi_\om^{(3)}(0)dP(\om)\right|=o(n^{-\frac12}).
\]
The term $\Pi_{\om,n}^{(3)}(0)/n$ is an ergodic average, but such fast rate of convergence in the strong law of large number is not even true in general for sums of independent and identically distributed random variables. However, we note that under certain mixing assumptions for the base map $\sigma$, the rate of order $n^{-\frac12}\ln n$ was obtained in \cite{HafEDG} (see also \cite{HeSt}).
 \end{rmk}

\begin{rmk}
We note that condition~\eqref{Large t's2} holds whenever~\eqref{Large t's} is satisfied. 
\end{rmk}

\begin{proof}[Proof of Theorem \ref{EdgeThm}]
The purpose of the following arguments is to prove the second statement of Theorem \ref{EdgeThm}, and the proof of the first statement (the self-normalized Berry-Esseen theorem) is a by-product of these arguments. In particular, we will be using Taylor polynomials of order three of the function $\Pi_{\om,n}(\cdot)$,  but it order to prove the self normalized Berry-Esseen theorem we could have used only second order approximations.

Let $t\in\bbR$. Then by (\ref{Basic relation}) and Lemma \ref{lem:UnifExpDecayY2} when $t_n=t/\sig_{n}$ is sufficiently small, uniformly in $\om$ we have
\begin{equation}\label{One}
\int e^{it_n \cdot S_n g(\om,\cdot)}d\mu_\om=\phi_\om^{it_n}(v_\om^0)e^{\Pi_{\om,n}(it_n)}+|t_n|\mathcal O(r^n).
\end{equation}
As in (\ref{Second order}), since $\phi_\om^{0}(v_\om^0)=1$ and the derivative of $z\to \phi_\om^{z}(v_\om^0)$ vanishes at $z=0$ we have 
\[
|\phi_\om^{it_n}(v_\om^0)-1|\leq Ct_n^2.
\]
Using Lemma~\ref{lem:boundProdLam} and that $\sig_n \sim n^{\frac 12}\Sig$, we conclude  that when $n$ is sufficiently large and  $t_n=t/\sig_n$ is sufficiently small,
\begin{equation}\label{two}
\left|\int e^{it_n \cdot  S_n g(\om,\cdot)}d\mu_\om-e^{\Pi_{\om,n}(it_n)}\right|\leq C\big(t_n^2e^{-ct^2}+|t_n|r^n\big).
\end{equation}
where $c,C>0$ are some constant. Next, by considering the function $g(t)=e^{zt}$, where $z$ is a fixed complex number, we derive that 
\begin{equation}\label{e z}
\left|e^{z}-1-z\right|\leq |z|^2e^{\max\{0,\Re(z)\}}.
\end{equation}
Since $\sig_n \sim n^{\frac 12}\Sig$, Lemma~\ref{lem:boundProdLam} together with the fact that $\Sigma>0$ yields that $\Re(\Pi_{\om,n}(it_n))\leq -ct^2$ when $|t_n|$ is sufficiently small and $n$ is large enough, where $c>0$ is a constant which does not depend on $\om$, $t$ and $n$ (we can clearly assume that $c<\frac12$). It follows that  $\max\{0,\Re(t^2/2+\Pi_{\om,n}(it_n))\}\leq (\frac 12-c)t^2$. Applying (\ref{e z}) with $z=t^2/2+\Pi_{\om,n}(it_n)$ yields that when $n$ is sufficiently large and $|t_n|$ is sufficiently small,
\begin{equation}\label{three}
\left|e^{\Pi_{\om,n}(it_n)}-e^{-\frac 12 t^2}(1+\Pi_{\om,n}(it_n)+\frac12 t^2)\right|
=e^{-\frac12 t^2}|e^{z}-1-z|\leq e^{-ct^2}\left|\Pi_{\om,n}(it_n)+\frac12 t^2\right|^2.
\end{equation}
Next, using the formula for Taylor reminder of order $3$, we have that 
\begin{equation}\label{Four}
|\Pi_{\om,n}(it_n)+\frac12 t^2-\cP_{\om,n}(t)|\leq Cnt_n^4.
\end{equation}
Observe also that 
\[
\cP_{\om,n}(t)=\frac12t^2\big(1-\Pi_{\om,n}''(0)/\sig_n^2\big)-\Pi_{\om,n}'''(0)t^3/\sig_n^3.
\]
The second term on the right hand side is $\mathcal O(|t|^3)n^{-\frac12}$ while by
 (\ref{D2 Pi}) the first term is $\mathcal O(t^2/\sig_n^2)=\mathcal O(t^2)/n$. We conclude that 
\begin{equation}\label{Five}
\left|\Pi_{\om,n}(it_n)+\frac12 t^2\right|\leq C\max(t^2,|t|^3)n^{-\frac12}
\end{equation}
and hence 
\begin{equation}\label{Six}
\left|e^{\Pi_{\om,n}(it_n)}-e^{-\frac 12 t^2}(1+\Pi_{\om,n}(it_n)+\frac12 t^2)\right|
=e^{-\frac12 t^2}|e^{z}-1-z|\leq Ce^{-ct^2}\max(t^4,t^6)/n.
\end{equation}
From~\eqref{Four} and~\eqref{Six},
we conclude that
\begin{equation}\label{Seven}
\left|e^{\Pi_{\om,n}(it_n)}-e^{-\frac 12 t^2}(1+\cP_{\om,n}(t))\right|
\leq C''e^{-ct^2}\max(t^4,t^6)/n.
\end{equation}

Finally, using the Berry-Esseen inequality we derive that 
\begin{equation}\label{Tch}
\sup_{t\in\mathbb R}\left|\mu_\omega(\{S_ng(\omega,\cdot)\leq t\sigma_n\})-A_{\omega,n}(t)\right|\leq 
\int_0^T\bigg{\lvert}\frac{\mu_\omega(e^{\frac{it \cdot S_ng(\omega, \cdot)}{\sigma_n}})-e^{-\frac 1 2 t^2}(1+\cP_{\om,n}(t))}{t}\bigg{\rvert}\, dt +C/T,
\end{equation}
where $C$ is some constant.  
We have used here the fact that the derivative of $A_{\om,n}$ is bounded by some constants (since the coefficients of the polynomial $\cP_{\om,n}$ are bounded in $\om$ and $n$). In order to establish the first assertion of the theorem, we choose $T$ of the form $T=\del_0\sqrt n$, where $\del_0>0$ is sufficiently small. Indeed, observe that the above estimates imply that 
\[
\sup_{t\in\mathbb R}\left|\mu_\omega(\{S_ng(\omega,\cdot)\leq t\sigma_n \})-A_{\om,n}(t)\right|=\mathcal O(n^{-\frac12}).
\]
Set $\varphi(t)=\frac1{\sqrt{2\pi}}e^{-\frac 12 t^2}$.
Integrating both sides of the equation $A'_{\om,n}(t)=Q_{\om,n}(t)e^{-\frac12 t^2}$, where $Q_{\om,n}$ satisfies (\ref{Q om n}) and using that $a_{\om,n}=\mathcal O(1/n)$,  we conclude that
\begin{equation}\label{BE0}
\sup_{t\in\mathbb R}\left|\mu_\omega(\{S_ng(\omega,\cdot)\leq t\sigma_n \})-\Phi(t)-u_{\om,n}\sig_n^{-1}(t^2-1)\varphi(t)\right|=O(n^{-\frac12}).
\end{equation}
Recall that $u_{\om,n}=\frac{\Pi_{\om,n}^{(3)}(0)}{\sig^2_n}$, which converges to $\Sig^{-2}\int \Pi_\om^{(3)}(0)dP(\om)$ as $n\to\infty$, and in particular it is bounded. Therefore, $\sup_{t\in \R}|u_{\om,n}\sig_n^{-1}(t^2-1)\varphi(t)|=\mathcal O(n^{-\frac12})$,
which together with (\ref{BE0}) yields (\ref{SelfNor BE}).

Next, in order to prove the second item, 
fix some $\ve>0$ and choose $T$ of the form $C/T={\ve}n^{-\frac12}$. We then have that
\[
\begin{split}
& \int_0^T\bigg{\lvert}\frac{\mu_\omega(e^{\frac{it \cdot S_ng(\omega, \cdot)}{\sigma_n}})-e^{-\frac 1 2 t^2}(1+\cP_{\om,n}(t))}{t}\bigg{\rvert}\, dt +C/T \\
&\le \int_0^{\del_0 \sqrt n}\bigg{\lvert}\frac{\mu_\omega(e^{\frac{it \cdot S_ng(\omega, \cdot)}{\sigma_n}})-e^{-\frac 1 2 t^2}(1+\cP_{\om,n}(t))}{t}\bigg{\rvert}\, dt \\
&\phantom{\le}+\int_{\del_0 \sqrt n \leq |t|\leq \frac{C}{\ve}\sqrt n }\bigg{\lvert}\frac{\mu_\omega(e^{\frac{it \cdot S_ng(\omega, \cdot)}{\sigma_n}})-e^{-\frac 1 2 t^2}(1+\cP_{\om,n}(t))}{t}\bigg{\rvert}\, dt 
+\ve n^{-\frac12}.
\end{split}
\]
Using (\ref{Seven}) we see that the first integral on the above right hand side is of order $\mathcal O(n^{-1})$, while the second integral is $o(n^{-\frac12})$ by (\ref{Large t's2}).
\end{proof}
\begin{rmk}
In \cite{HafEDG} expansions of order larger than $1$ were obtained for some classes of interval maps under the assumption that the modulos of the characteristic function $\varphi_n(t)$ of $S_n g(\om,\cdot)$ does not exceed $n^{-r_1}$ when $|t|\in[K,n^{r_2}]$, where $K,r_1,r_2$ are some constants.
Of course, under such conditions we can obtain higher order expansions also in our setup, but since we do not have examples under which this condition holds true (expect from the example covered in \cite{HafEDG}),  the proof (which is  very close to \cite{HafEDG}) is omitted.
\end{rmk}
\subsubsection{Some asymptotic expansions for large deviations}\label{LDex}
In this section we again consider the 
scalar case when $d=1$. We will also assume that there exist constants $C_1,C_2,r>0$ so that for $\mathbb P$-a.e. $\om \in \Om$,   $z\in \C$ with $|z|\leq r$ and a sufficiently large $n\in \N$ we have
\begin{equation}\label{RPF}
C_1\leq \|\mathcal L_\om^{z,(n)}\| /|\lam_\om^{z,(n)}|\leq C_2,
\end{equation}
where $\lam_\om^{z, (n)}=\prod_{i=0}^{n-1}\lam_{\sigma^i \om}^z$.
Moreover,  we assume that there exists a constant $C>0$ such  that  for $\mathbb P$-a.e. $\om \in \Om$ and for any $t,s\in \R$, we have that 
\begin{equation}\label{LY}
\sup_{n\in\mathbb N}\|\cL_{\om}^{\te+is,(n)}\|/\|\cL_{\om}^{\te,(n)}\|\leq C(1+|\te|+|s|).
\end{equation}
These conditions are satisfied in the setup of~\cite[Chapter 5]{HK} (the second condition follows from the arguments in the Lasota-Yorke inequality which was proved in~\cite[Lemma 5.6.1]{HK}).

Our results in this subsection  will rely on the following lemma.
\begin{lemma}\label{Per}
Suppose that: 
\begin{enumerate}
\item  $\mathcal F$ is the Borel $\sig$-algebra on $\Om$;
\item $\sigma$ has a periodic point $\om_0$ (whose period is denoted by $n_0$), and $\sigma$ is continuous at each point that belongs to the orbit of $\om_0$;
\item $\mathbb P(U)>0$ for any open set $U$ that intersects the orbit  of $\om_0$;
\item  for any compact set $K\subset \R$ the family of  maps $\om\to\cL_\om^z,\,z\in K$ is uniformly continuous at the orbit points of $\om_0$;
\item  for any sufficiently small $\theta$ and $s\not=0$ the spectral radius of $\cL_{\om_0}^{\theta+is,(n_0)}$ is smaller than the  spectral radius of $\cL_{\om_0}^{\theta,(n_0)}$. 
\end{enumerate}
Then,  there exists $r>0$ with the following property: 
for $\mathbb P$-a.e. $\om$ and for any compact set $J\subset\bbR\setminus\{0\}$ there exist constants  $C_J(\om)$ and $c_J(\om)>0$ so that  for any sufficiently large $n$, $\theta\in [-r,r]$ and $s\in J$ we have
\[
\|\cL_{\om}^{\theta+is,(n)}\|\leq C_J(\om)e^{-c_J(\om) n}\|\cL_{\om}^{\theta,(n)}\|
\] 
\end{lemma}

\begin{proof}
Denote by $	r(z),\,z\in\bbC$ the spectral radius of the deterministic transfer operator $\mathcal R_{z}:=\mathcal L_{\om_0}^{z,(n_0)}$. Let $J\subset\bbR\setminus\{0\}$ be a compact set. Since $\mathcal R_{z}$ is continuous in $z$ and $r(\te)$ is continuous around the origin, there exist $\del,d_0>0$ which depend on $J$ so that for any $\te\in[-r,r]$, $s\in J$ and $d\geq d_0$ we have
\[
\|\mathcal R_{\te+is}^d\|\leq (1-\del)^{d}r(\te)^d.
\]
Observe that we have also taken into account the last assumption in the statement of the lemma. 
Note that a deterministic version of (\ref{RPF}) holds true with the operators $\mathcal R_z$ and thus there is a constant $C>0$ such that 
\[
\|\mathcal R_\te^d\|\geq C r(\te)^d
\]
for any $\te\in[-r,r]$.
Let $K\subset\bbR$ be a bounded closed interval around the origin which contains $J$.
Fix some $d>d_0$ and let $\ve\in(0,1/2)$ and $\om_1\in\Om$ be so that 
\begin{equation}\label{U def}
\|\mcl_{\om_1}^{\te+is, (dn_0)}-\mcl_{\om_0}^{\te+is, (dn_0)}\|=\|\mcl_{\om_1}^{\te+is, (dn_0)}-\mathcal R_{\te+is}^{d}\|<\varepsilon \min\{r(\te)^d,1\},
\end{equation}
for any $\te\in[-r,r]$ and $s\in K$.
By (\ref{RPF}) we have
\begin{equation}\label{EQ1}
0<C_1\le \|\mcl_{\om}^{\te, (n)}\|/ \lvert \lam_{\om}^{\te, (n)}\rvert \leq C_2<\infty,
\end{equation}
for some constants $C_1$ and $C_2$ which do not depend on $\om$ and $n$. Therefore, if $\ve$ is small enough then
\begin{equation}\label{Ratio}
1/\lvert \lam_{\om_1}^{\te,(dn_0)} \rvert \leq C/(\|\mathcal R_{\te}^{d}\|-\varepsilon \min\{r(\te)^d,1\})\leq 
C'/(r(\te)^d-\varepsilon \min\{r(\te)^d,1\}),
\end{equation}for some constants  $C, C'>0$. 
We conclude that
\begin{equation}\label{WeConc}
\|\cL_{\om_1}^{\te+is, (dn_0)}\|/ \lvert \lam_{\om_1}^{\te, (dn_0)} \rvert \leq C\,\frac{\ve \min\{r(\te)^d,1\}+(1-\del)^{d}r(\te)^d}{r(\te)^d-\varepsilon \min\{r(\te)^d,1\}}\leq C''(\ve+(1-\del)^d),
\end{equation}
where $C''>0$ is another constant.
By (\ref{LY})  we have  that
\begin{equation}\label{B J}
\esssup_{\om \in \Om}\sup_{s\in J,\te\in[-r,r],n\in\mathbb N}\|\cL_{\om}^{\te+is,(n)}\|/\|\cL_{\om}^{\te,(n)}\|\leq B_J<\infty,
\end{equation}
for some constant $B_J$ which depends only on $J$. Fixing a sufficiently large $d$ and then a sufficiently small $\ve$ we conclude that for any $\te\in[-r,r]$, $s\in J$ and $n$ we have that 
\[
\frac{\|\cL_{\om}^{\te+is, (n)}\cL_{\om_1}^{\te+is,(dn_0)}\|}{|\lam_\om^{\te,(n)}\lam_{\om_1}^{\te,(dn_0)}|}\leq C_2 C'' B_J(\ve+(1-\del)^d)<\frac12.
\]
Indeed, 
\begin{eqnarray*}
\|\cL_{\om}^{\te+is,(n)}\cL_{\om_1}^{\te+is,(dn_0)}\|\leq \|\cL_{\om}^{\te+is,n}\|\cdot\|\cL_{\om_1}^{\te+is,(dn_0)}\|\leq B_J\|\cL_{\om}^{\te,(n)}\|\cdot |\lam_{\om_1}^{\te,(dn_0)}|C''(\ve+(1-\del)^d)\\\leq 
B_JC_2 |\lam_{\om}^{\te,(n)}|\cdot |\lam_{\om_1}^{\te,(dn_0)}|C''(\ve+(1-\del)^d),
\end{eqnarray*}
where in the first inequality we have used the submultiplicativity of operator norm, in the second we have used (\ref{WeConc}) and (\ref{B J}) and in the third one we have used (\ref{EQ1}).

Finally,  because of the fifth condition in the statement  of the lemma and since $r(\te)$ is continuous in $\te$ (around the origin), when $r$ is small enough we have that (\ref{U def}) holds true for any $\om_1\in U$, $\te\in[-r_0,r_0]$ and $s\in K$ and , where $U$ is a sufficiently small open neighborhood of the periodic point $\om_0$ and $r_0$ depends only on the function $r(\te)$. By ergodicity of $\sig$,\, for  $\mathbb P$-a.e. $\om \in \Om$ we have an infinite strictly increasing sequence $a_n=a_n(\om)$ of visiting times to $U$ so that $a_n/n$ converges to $1/P(U)$ as $n\to\infty$. Thus, by considering the subsequence $b_n=a_{ndn_0} (\om)$ 
 we can write $\cL_{\om}^{\te+is,(n)}$ as composition of blocks of the form $\cL_{\om'}^{\te+is,(m)}\cL_{\om_1}^{\te+is,(dn_0)}$ (and perhaps a single block of the form $\cL_{\om''}^{\te+is,(m)})$, where $m\geq0$ and $\om_1\in U$. The number of blocks is approximately $nP(U)/dn_0$ (i.e. when divided by $n$ it converges to $P(U)/dn_0$ as $n\to\infty$). 
  Therefore, 
\[
\|\cL_{\om}^{\te+is,(n)}\|\leq C_J(\om)|\lam_{\om}^{\te,n}|2^{-nP(U)/2dn_0},
\]
which together with (\ref{EQ1}) completes the proof of the lemma.
\end{proof}

Our main result here is the following theorem. 
\begin{thm}\label{LD EX}
Suppose that the conclusion of Lemma \ref{Per} holds true and that $\Sig^2>0$. Then for any sufficiently small $a>0$ and  for $\mathbb P$-a.e.  $\om \in \Om $ we have
\[
\mu_\om (\{x: S_n g(\om,x)\geq an\})\cdot e^{nI_{\om,n}(a)}=\frac{\phi_\om^{\te_{\om,n,a}}(v_\om^0)\sqrt{I_{\om,n}''(a)}}{\te_{\om,n,a}\sqrt{2\pi n}}(1+o(1)).
\]
Here, 
\[
I_{\om,n}(a)=\sup_{t\in[0,r]}(t\cdot a-\Pi_{\om,n}(t)/n)=\te_{\om,n,a}-\Pi_{\om,n}(\te_{\om,n,a})/n,
\]
where $r>0$ is any sufficiently small number.
\end{thm}
\begin{rmk}
Set 
\[
I(a)=\sup_{t\in[0,r]}(t\cdot a-\Lambda(t))=\te_a-\Lambda(\te_a).
\]
Then 
\[
\lim_{n\to\infty} I_{\om,n}(a)=I(a)\,\,\text{ and }\,\,\lim_{n\to\infty} \te_{\om,n,a}=\te_a.
\]
Furthermore, we have that $\lim_{n\to\infty} I''_{\om,n}(a)=I''(a)$ (using the duality of Fenchel-Legandre transforms).
\end{rmk}

\begin{proof}
The proof follows the general scheme used in the proof of~\cite[Theorem 2.2]{FH} together with arguments similar to the ones in the proof of Theorem~\ref{EdgeThm}. Therefore, we will only provide a sketch of the arguments. 
Let $a$ be sufficiently small.
Denote by $F_n^\om$ the distribution of $S_n g(\om,\cdot)$ and set
\[
d\tilde F_{\om,n}(x)=\left(e^{\theta_{\om,n,a}x}/\lam_{\om}^{\theta_{\om,n,a},(n)}\right) \,dF_{\om,n}(x).
\]  
Note that $d\tilde F_{\om,n}$ is a finite measure, which in general is not a probability measure. Set $G_{\om,n}(x)=\tilde F_{\om,n}((-\infty,x\sqrt n+an])$. Arguing as in the proof of~\cite[Theorem 2.3]{FH} (and using the consequence of Lemma \ref{Per}), it is enough to show that the non-normalized distribution functions $G_{\om,n}$ admit Edgeworth expansions of order $1$ (see Lemmas 3.2 and 3.3 in~\cite{FH}). Observe that (when $|s|$ is sufficiently small),
\begin{equation}\label{hat G}
\hat G_{\om,n}(s\sqrt n)=
(e^{-isna}/\lam_{\om}^{\te_{\om,n,a},(n)})\int e^{(\te_{\om,n,a}+is) S_n g(\om,\cdot)}d\phi_\om^0
=\bar\mu_{\om,n}(s)\phi_{\om}^{\te_{\om,n,a}+is}(v_\om^0)+\del_{\om,n}(s),
\end{equation}
 where
\[
\begin{split}
\bar\mu_{\om,n}(s) &=e^{-iasn}\lam_\om^{\te_{\om,n,a}+is,(n)}/\lam_\om^{\te_{\om,n,a},(n)}\\
&=e^{\Pi_{\om,n}(\te_{\om,n,a}+is)-\Pi_{\om,n}(\te_{\om,n,a})-ians}\\
&=e^{\Pi_{\om,n}(\te_{\om,n,a}+is)-\Pi_{\om,n}(\te_{\om,n,a})-i\Pi_{\om,n}'(0)s},
\end{split}
\]
and $\del_{\om,n}(z)$ is an holomorphic function of $z$ such that, uniformly in $\om$ we have  $\del_{\om,n}(z)=\mathcal O(r^n)$ for some $r\in(0,1)$ (and hence all of the derivatives of $\del_{\om,n}$ at zero are at most of the same order). By arguing as in the proof of Theorem \ref{EdgeThm} we obtain Edgeworth expansions of order $1$ for $G_{\om,n}$. 
\end{proof}

\section{Hyperbolic dynamics}
The purpose of this section is to briefly discuss and indicate that  almost all of our main results can be extended to the class of random hyperbolic dynamics introduced in~\cite[Section 2.]{DFGTV2}. We stress that the spectral approach developed in~\cite{DFGTV1} for the random piecewise expanding dynamics has been extended to the 
random hyperbolic case in~\cite{DFGTV2} for the real-valued observables. By combining techniques developed in the present paper together with those in~\cite{DFGTV2}, we can now treat the case of vector-valued observables. In addition, we are not only able to provide the versions of the results in~\cite[Sections 7 and 8]{DFGTV2} for vector-valued observables but we can also establish versions of  almost all other results covered in the present paper (that have not been established previously even for real-valued observables).  

Let $X$ be a finite-dimensional $C^\infty$ compact connected Riemannian manifold. Furthermore, let $T$ be a topologically transitive Anosov diffeomorphism of class $C^{r+1}$ for $r>2$. As before, let $(\Om, \mathcal F, \mathbb P)$ be a probability space such that $\Om$ is a Borel subset of a separable, complete metric space. Furthermore, let $\sig \colon \Om \to \Om$ be a homeomorphism. As in~\cite[Section 3]{DFGTV2}, we now build a cocycle $(T_\om)_{\om \in \Om}$ such that all $T_\om$'s are Anosov diffeomorphisms that belong to a sufficiently small neighborhood of $T$ in the  $C^{r+1}$ topology on $X$. Furthermore, we require that $\om \to T_\om$ is 
measurable. Let $\mcl_\om$ be the transfer operator associated to $T_\om$. It is was verified in~\cite[Section 3]{DFGTV2} that conditions (C0) and (C2)-(C4) hold, with:
\begin{itemize}
\item $\mathcal B=(\mathcal B, \lVert \cdot \rVert_{1,1})$ is the space $\mathcal B^{1,1}$ which belongs to the class of anisotropic Banach spaces introduced by Gou{\"e}zel and  Liverani~\cite{GL}. We stress that in this setting the second altenative in (CO) holds. Namely, $\mathcal B$ is separable and the cocycle of transfer operators is strongly measurable;
\item (C3) holds with constant $\alpha^N$ and $\beta^N$.
\end{itemize}
We recall that elements of $\mathcal B$ are distributions of order $1$. By  $h(\varphi)$ we will denote the action of $h\in \mathcal B$ on a test function $\varphi$.
 We note that in this setting, it was proved in~\cite[Lemma 3.5. and Proposition 3.6.]{DFGTV2} that the version of Lemma~\ref{lem:qc+1dim} holds true. Moreover, one can show (see~\cite[Proposition 3.3. and Proposition 3.6.]{DFGTV2}) that the top Oseledets space $Y(\om)$ is spanned by a Borel probability measure $\mu_\om$ on $X$.

We now consider a suitable class of observables. 
Let us fix a measurable map $g\colon \Om \times X\to \R^d$  such that:
\begin{itemize}
\item $g(\om, \cdot)\in C^r$ and $\esssup_{\om \in \Om} \lVert g(\om, \cdot)\rVert_{C^r}<\infty$;
\item for $\mathbb P$-a.e. $\om \in \Om$ and $1\le i\le d$, 
\[
\int_X g^i(\om, \cdot) \, d\mu_\om =0.
\]
\end{itemize}
We recall (see~\cite[p. 634]{DFGTV2}) that for $h\in \mathcal B$ and $g\in C^r(X, \C)$ we can define $g\cdot h\in \mathcal B$. Furthermore, the action of $g\cdot h$ as a distribution is given by 
\[
(g\cdot h)(\varphi)=h(g\varphi), \quad \varphi \in C^1(X, \mathbb C).
\]
This enables us to introduce twisted transfer operators. Indeed, for $\theta \in \C^d$ we introduce $\mcl_\om^\theta \colon \mathcal B\to \mathcal B$ by
\[
\mcl_\om^\theta h=\mcl_\om (e^{\theta \cdot g(\om, \cdot)}\cdot h), \quad h\in \mathcal B.
\]
By arguing as in the proof of~\cite[Proposition 4.3.]{DFGTV2}, one can establish the version of Lemma~\ref{l49} in this setting. 

Let us now introduce  appropriate versions of spaces $\mathcal S$ and $\mathcal S'$ from Section~\ref{lar}  in the present context. Let $\mathcal S'$ denote the space of all measurable maps $\mathcal V\colon \Omega \to \mathcal B$ such that 
\[
\lVert \mathcal V\rVert_\infty:=\esssup_{\omega \in \Omega}\lVert \mathcal V(\omega)\rVert_{1,1}<\infty.
\]
Then, $(\mathcal S', \lVert \cdot \rVert_\infty)$ is a Banach space. Let $\mathcal S$ consist of all $\mathcal V\in \mathcal S'$ with the property that $\mathcal V(\omega)(1)=0$ for $\mathbb P$-a.e.  $\omega \in \Omega$, where $1$ denotes the observable taking the value $1$ at all points. Then, $\mathcal S$ is a closed
subspace of $\mathcal S'$ (see~\cite[p.641]{DFGTV2}). 

For $\theta \in \mathbb C^d$ and $\mathcal W\in \mathcal S$, set 
\[
F(\theta, \mathcal W)(\omega)=\frac{\mathcal L_{\sigma^{-1}\omega}^\theta (\mathcal W(\sigma^{-1}\omega)+\mu_{\sigma^{-1}\omega})}{\mathcal L_{\sigma^{-1}\omega}^\theta (\mathcal W(\sigma^{-1}\omega)+\mu_{\sigma^{-1}\omega})(1)}-\mathcal W(\omega)-\mu_\omega, \quad \omega \in \Omega.
\]
By arguing as in the proofs of Lemma~\ref{lem:FwellDef} and~\cite[Lemma 5.3.]{DFGTV2}, we find that $F$ is a well-defined  analytic map on $\mathcal D=\{\theta \in \mathbb C^d: |\theta| \le \epsilon \} \times B_{\mathcal S}(0, R)$ for some $\eps, R>0$, where $B_{\mathcal S}(0, R)$ denotes the open ball in $\mathcal S$ of radius $R$ centered at the origin.

The following is a version of Lemma~\ref{thm:IFT} in the present setting. 
\begin{lemma}
By shrinking $\epsilon>0$ if necessary, we have that there exists $O\colon \{ \theta \in \C^d: |\theta|<\ep \} \to \mc S$ analytic in $\theta$ such that 
\begin{equation}
F(\theta, O(\theta))=0.
\end{equation}
\end{lemma}

\begin{proof}
We first note that (see~\cite[p.636]{DFGTV2}) that there exist $D, \lambda >0$ such that 
\[
\lVert \mcl_\om^{(n)} h\rVert_{1,1}\le De^{-\lambda n}\lVert h\rVert_{1,1}, \quad \text{for $h\in \mathcal B$, $h(1)=0$ and $n\in \N$.}
\]
Moreover, the same arguments as in the proof of Proposition~\ref{difF} (see also~\cite[Proposition 5.4.]{DFGTV2}) yield that 
\[
(D_{d+1}F(0,0) \mathcal W)(\om)=\mcl_{\sigma^{-1}\om}\mathcal W(\sigma^{-1}\om)-\mathcal W(\om), \quad \text{for $\om \in \Om$ and $\mathcal W\in \mathcal S$.}
\]
Now by arguing exactly as in the proof of Lemma~\ref{thm:IFT} we conclude that $D_{d+1}F(0, 0)$ is invertible and thus the desired  conclusion follows from the implicit function theorem. 
\end{proof}
Let $\Lam(\theta)$ be the largest Lyapunov exponent associated to the twisted cocycle $\mcl^\theta=(\mcl_\om^\theta)_{\om \in \Om}$. 
Let
\[
\mu_\om^\theta := \mu_\om +O(\theta)(\om), \quad \text{for $\theta \in \mathbb C^d$, $|\theta|<\eps$.}
\]
Observe that  $\mu_\om^\theta(1) =1$ and by the previous lemma, $\theta \mapsto \mu_\om^\theta$ is analytic.
 Let us define
\[
 \hat\Lambda (\theta) :=  \int_\Om \log \Big|\mu_\om^\theta( e^{\theta \cdot  g(\om, x)} ) \Big|\, d\bbp(\om),
\]
and
\[
\lot :=  \mu_\om^\theta( e^{\theta \cdot  g(\om, x)})
=( \mcl_\om^\theta \mu_\om^\theta)(1).
\]
The proof of the following result is analogous to the proof of~\cite[Lemma 6.1.]{DFGTV2} (see also the Lemmas in Section~\ref{LMB}).
\begin{lemma}
\begin{enumerate}
\item[(1)] For every $\theta \in B_{\C^d}(0,\ep):= \{ \theta \in \C : |\theta|<\ep \}$, we have $ \hat\Lambda (\theta)\leq \Lambda (\theta)$.
\item[(2)] $\hat\Lambda$ is differentiable on a neighborhood of 0, and for each $i\in \{1, \ldots, d\}$ we have that
\[
D_i\hat \Lambda (\theta)=
\Re \Bigg( \int_\Om \frac{ \overline{\lot}  ( \mu_\om^\theta( g^i(\om, \cdot)e^{\theta \cdot g(\om, \cdot)})+  (D_i O(\theta)) (\om) (e^{\theta \cdot  g(\om, \cdot)}) )}{|\lot |^2}\, d\bbp(\om) \Bigg),
\]
where $D_i$ denotes the derivative with respect to $i$-th component of $\theta$.
\item[(3)]For $i\in \{1, \ldots, d\}$, 
 we have that $D_i \hat \Lambda(0)=0$.
\end{enumerate}
\end{lemma}

\begin{lemma}
\begin{enumerate}
\item[(1)] For $\theta \in \mathbb C^d$ sufficiently close to $0$, the twisted cocycle $\mathcal L^\theta=(\mcl_\om^\theta)_{\om \in \Om}$ is quasi-compact. Furthermore, the top Oseledets space of $\mathcal L^\theta$ is one-dimensional.
\item[(2)] The map $\theta \mapsto \Lambda (\theta)$ is differentiable near $0$ and $D_i \Lambda (0)=0$ for $i\in \{1, \ldots, d\}$.
\end{enumerate}
\end{lemma}

\begin{proof}
The quasi-compactness of $\mathcal L^\theta$ for $\theta$ close to $0$, as well as one dimensionality of the associated top Oseledets space, can be obtained by repeating the arguments in the proof of~\cite[Theorem 3.12]{DFGTV1} (which require the Lasota-Yorke inequalities obtained in~\cite[Lemma 3]{DH}).
Furthermore, the same argument as in the proof of~\cite[Corollary 3.14]{DFGTV1} implies that $\Lambda$ and $\hat{\Lambda}$ coincide on a neighborhood of $0$, which gives the second statement of the lemma. 
\end{proof}
By~\cite[Proposition 2.]{DH}, we have that there exists a positive semi-definite $d\times d$  matrix $\Sigma^2$  such that for $\mathbb P$-a.e. $\om \in \Om$, \eqref{VAR} holds. Furthermore, the elements of $\Sigma^2$ are given by~\eqref{varelem}.

The following is a version of Lemma~\ref{lem:Lam''0} in the present context. 
\begin{lemma}
We have that $\Lambda$ is of class $C^2$ on a neighborhood of $0$ and  $D^2 \Lam(0)=\Sig^2$, where $D^2\Lam (0)$ denotes the Hessian of $\Lambda$ in $0$.
\end{lemma}

\begin{proof}
The proof is completely analogous to that of Lemma~\ref{lem:Lam''0} and thus we only point out the small adjustments that need to be made. Namely, in the present context we have that 
\[ 
D_i \lam_\om^\theta= \mu_\om^\theta (g^i (\om, \cdot)e^{\theta \cdot g(\om, \cdot)}) +D_i O(\theta) (\om)  (e^{\theta \cdot g(\om, \cdot) }),
\]
and
\[
\begin{split}
D_{ij}\lam_\om^\theta &= \mu_\om^\theta (g^i (\om ,\cdot)g^j (\om, x)e^{\theta \cdot g(\om, \cdot)} ) + D_i O(\theta) (\om)  (g^j (\om, \cdot)e^{\theta \cdot g(\om, \cdot)})  \\
&\phantom{=}+ D_j O(\theta) (\om ) (g^i (\om, \cdot)e^{\theta \cdot g(\om, \cdot)} )+ D_{ij}O(\theta) (\om) (e^{\theta \cdot g(\om, \cdot)}),
\end{split}
\]
for $1\le i, j\le d$.
Due to the centering condition for $g$ and the fact that $D_iO(0)\in \mathcal S$, we have that $D_i \lam_\om^\theta \rvert_{\theta=0}=0$ for $1\le i\le d$. In addition, since $D_{ij}O(0)\in \mathcal S$ we have that 
\[
D_{ij}\lam_\om^\theta \rvert_{\theta=0} =\mu_\om( g^i (\om ,\cdot)g^j (\om, \cdot)) 
+  D_i O(0) (\om)  (g^j (\om, \cdot))+D_j O(0) (\om ) (g^i(\om, \cdot)),
\]
and therefore
\[
\begin{split}
D_{ij} \Lam (0) &= \Re \bigg (\int_{\Omega \times X} g^i (\om ,x)g^j (\om, x)\, d\mu(\om, x) +\int_{\Om}   D_i O(0) (\om)  (g^j (\om, \cdot)) \, d\mathbb P(\om) \\
&\phantom{=}+\int_{\Om}   D_j O(0) (\om)  (g^i (\om, \cdot))\, d\mathbb P(\om) \bigg ),
\end{split}
\]
for $1\le i, j\le d$. The rest of the proof proceeds exactly as the proof of Lemma~\ref{lem:Lam''0}, by taking into account that 
 \[
D_i O(0) (\om)=\sum_{n=1}^\infty \mathcal L_{\sigma^{-n}  \om}^{(n)} (g^i(\sigma^{-n} \om, \cdot) \cdot \mu_{\sigma^{-n} \om}), \quad 1\le i\le d.
\]
\end{proof}
Now the choice for the  bases for top Oseledets spaces $Y_\omega ^\theta$ and $Y_\omega ^{*\theta}$ can be made as in Subsection~\ref{sec:choiceOsBases}.
\subsection{Limit theorems}
In the preceding discussion we have established all preparatory material (analogous to that for piecewise expanding case) for limit theorems in the context of random hyperbolic dynamics. 
The following is a version of Lemma~\ref{L:growthExpSums} in the present context. The proof is again the same as the proof of~\cite[Lemma 4.2.]{DFGTV1} (and relies only on the Oseledets decomposition). We sketch it for readers' convenience.
\begin{lemma}
Let $\theta\in \C^d$ be sufficiently close to 0. Furthermore, 
let $h\in \mathcal B$ be such that  $\phi^\theta_\omega (h) \neq 0$.
Then,
\[
\lim_{n\to\infty}\frac{1}{n} \log \Big| h( e^{\theta \cdot  S_ng(\omega,\cdot )}) \Big| =  \Lam(\theta) \quad \text{for \paeom.}
\]
\end{lemma}
\begin{proof}
We use the notation of Subsection~\ref{sec:choiceOsBases} adapted to the present setting. Given $h\in \B$, we write $h=\phi^\theta_\omega (h) \mu^\theta_\omega+h^\theta_\omega$, where $h^\theta_\omega\in H^\theta_\omega$.
Then,
\[
\label{decomp}\mathcal{L}^{\theta,(n)}_\omega h=\left(\prod_{i=0}^{n-1}\lambda_{\sigma^i\omega}^\theta\right) \phi^\theta_\omega (h)  \mu^\theta_{\sigma^{n-1}\omega}+ \mathcal{L}^{\theta,(n)}_\omega h_\omega^\theta.
\]
By the multiplicative ergodic theorem, we have for $\mathbb P$-a.e. $\om \in \Om$ that 
\begin{equation}
\label{decay}
\lim_{n\to\infty}\frac{1}{n}\log\|\mathcal{L}^{\theta,(n)}_{\omega}|_{H_\omega^\theta}\|<\Lambda(\theta).
\end{equation}
Thus, we have that for \paeom\, (since $\phi^\theta_\omega (h) \neq 0$),
\[
\begin{split}
\lim_{n\to\infty}\frac{1}{n} \log \Big| h(  e^{\theta \cdot S_ng(\omega,\cdot)} ) \Big| &= \lim_{n\to\infty}\frac{1}{n}\log \Big|  \mathcal{L}^{\theta,(n)}_\omega h (1) \Big| \\
&= \max  \bigg\{ \lim_{n\to\infty}\frac{1}{n}\sum_{i=0}^{n-1}\log |\lambda_{\sigma^i\omega}^\theta|, \lim_{n\to \infty} \frac 1 n \log | \mathcal{L}^{\theta,(n)}_\omega h_\omega^\theta (1) | \bigg{\}}\\
&=\Lambda(\theta), 
\end{split}
\]
where in the last step we have used~\eqref{decay} and the equality
\[
\Lambda (\theta)=\lim_{n\to \infty} \frac 1 n \lim_{n\to \infty} \lVert \mcl_\om^{\theta, (n)} \mu_\om^\theta \rVert_{1,1}=
\lim_{n\to\infty}\frac{1}{n}\sum_{i=0}^{n-1}\log |\lambda_{\sigma^i\omega}^\theta|.
\]
\end{proof}
The previous lemma readily implies that the version of Theorem~\ref{LDthm} in the present context.   Morever, we have the following version of Theorem~\ref{MDthm}.
\begin{thm}
Let $(a_n)_n$ be a sequence in $\R$ such that $\lim_{n\to\infty}\frac{a_n}{\sqrt n}=\infty$ and 
$\lim_{n\to\infty}\frac{a_n}n=0$. Then
for $\mathbb P$-a.e.  $\omega \in \Omega$ and any $\theta\in\mathbb R^d$,  we have that
\[
\lim_{n\to\infty}\frac{1}{a_n^2/n}\log \mathbb  E[e^{\theta \cdot  S_ng (\om, \cdot)/c_n}]=\frac 1 2\theta^T\Sigma^2\theta,
\]
where $c_n=n/a_n$.
Consequently, when $\Sig^2$ is positive definite we have that:
\begin{enumerate}
\item[(i)] for any closed set $A\subset\mathbb R^d$,
\[
\limsup_{n\to\infty}\frac{1}{a_n^2/n}\log  \mu_\omega(\{S_n g(\om,\cdot)/a_n\in A\})\leq -\frac 1 2 \inf_{x\in A}x^T\Sigma^{-2} x;
\]
\item[(ii)] for any open set $A\subset\mathbb R^d$ we have
\[
\liminf_{n\to\infty}\frac{1}{a_n^2/n}\log  \mu_\omega(\{S_n g(\om,\cdot)/a_n\in A\})\geq-\frac 1 2 \inf_{x\in A}x^T\Sigma^{-2} x,
\]
where $\Sig^{-2}$ denotes the inverse of $\Sig^2$.
\end{enumerate}
\end{thm}

\begin{proof}
The proof proceeds exactly as the proof of Theorem~\ref{MDthm} by replacing~\eqref{Basic relation} with 
\[
\int_X e^{\theta \cdot S_ng(\omega,\cdot)}d\mu_\omega= \mcl_\om^{\theta, (n)} \mu_\om (1)=
\phi_\omega^{\theta}(\mu_\om)e^{\Pi_{\omega,n}(\theta)}+ \mathcal L^{\theta,(n)}_\omega(\mu_\omega-\phi_\omega^{\theta}(\mu_\om)\mu_\omega^\theta)(1).
\]\end{proof}

One can now establish the Berry-Esseen theorem, Edgeworth expansions, local CLT and large and moderate deviations  exactly as in the case of random piecewise expanding dynamics with almost identical proofs. We remark that Lemma \ref{Per0} holds true for general cocycles $\cL_\om^{it}$ acting on a Banach space (see ~\cite[Lemma 2.10.4]{HK}). We also note that (\ref{Bound J}) holds true in our case without any additional assumptions. Indeed, this follows exactly as in the scalar case \cite[Lemma 9.3]{DFGTV2} (see the arguments in the proof of \cite[Lemma 4]{DH}).

Regarding the exponential concentration inequalities, in the present setting we are currently not able to obtain the version of Proposition~\ref{ConcenProp}. The reason is that the proof of Proposition~\ref{ConcenProp} relies 
on the martingale approach. Currently there exists only one paper (namely~\cite{DMN}) that explores the martingale method in the context of anisotropic Banach spaces adapted to hyperbolic dynamics. However, it is restricted to the case of deterministic dynamics and it is not clear if the techniques can be extended to the case of
random dynamics. The other limit theorem  which we can not obtain for random Anosov maps is the large deviations type expansions (Theorem \ref{LD EX}). The issue here is that, in contrary to the case of expanding maps, it is not clear to us when the additional assumption (\ref{LY}) holds true.

\begin{rmk}
We emphasize that it was convenient for us to use the class of anisotropic Banach spaces introduced in~\cite{GL}, since we could refer to the previous work in~\cite{DFGTV2, DH}. In principle, one could use any class of separable (in the  non-separable case, we would need to restrict to the first alternative in (C0))  anisotropic Banach spaces which are stable under small perturbations: 
the anisotropic Banach spaces associated to two Anosov diffeomorphisms $T$ and $T'$ coincide if $T$ and $T'$ are sufficiently close.  We refer to~\cite{Baladi} for an excellent survey on anisotropic Banach spaces for hyperbolic dynamics, and to~\cite{BL} for yet another interesting class of spaces recently introduced.
\end{rmk}
\section{Acknowledgements}
We would like to express our gratitude to the anonymous referee for hers/his  constructive comments that helped us to improve our paper.
D.D. was supported in part by Croatian Science Foundation under the project
IP-2019-04-1239 and by the University of Rijeka under the projects uniri-prirod-18-9 and uniri-pr-prirod19-16.

\section{Appendix}
We define $G \colon B_{\mathbb C^d}(0, 1) \times \mathcal S \to \mathcal S'$ and $H \colon B_{\mathbb C^d} (0, 1) \times \mc{S} \to L^\infty (\Om)$ by
 \begin{equation}\label{GH}
  G(\theta, \mc{W})_\om=\mathcal L_{\sigma^{-1} \om}^{\theta}(\mc{W}_{\sigma^{-1} \om}+v_{\sigma^{-1} \om}^0) \quad \text{and} \quad  H(\theta, \mc{W})(\om)=\int \mathcal L_{\sigma^{-1} \om}^{\theta}(\mc{W}_{\sigma^{-1} \om}+v_{\sigma^{-1} \om}^0) \, dm.
 \end{equation}
  Writing $\theta=(\theta_1, \ldots, \theta_d)$, by $D_iG$ we will denote the partial derivative of $G$ with respect to $\theta_i$ for $1\le i\le d$. Furthermore, $D_{d+1}G$ will denote the partial derivative of $G$ with respect to $\mc{W}$. Analogous notation will be used when $G$ is replaced with $H$.
\begin{lemma}\label{l1}
For $(\theta, \mc{W})\in B_{\mathbb C^d}(0, 1) \times \mathcal S$, 
we have that
\[
(D_{d+1}G(\theta, \mc{W})\mc{H})_\om=\mathcal L_{\sigma^{-1}\om}^\theta \mc{H}_{\sigma^{-1} \om}, \quad \text{for $\mc{H}\in \mc{S}$ and $\om \in \Om$.}
\]
Furthermore, for $(\theta, \mc{W})\in  B_{\mathbb C^d} (0, 1) \times \mc{S}$, we have that 
\[
(D_{d+1}H(\theta, \mc{W})\mc{H})(\om)=\int \mathcal L_{\sigma^{-1}\om}^\theta \mc{H}_{\sigma^{-1} \om}\, dm, \quad \text{for $\mc{H}\in \mc{S}$ and $\om \in \Om$.}
\]

\end{lemma}

\begin{proof}
The desired formulas follow directly from the simple observation that $G$ and $H$ are affine in $\mc{W}$.
\end{proof}

The proof of the following result is similar to the proof of~\cite[Lemma B.6.]{DFGTV1}.
\begin{lemma}
For $(\theta, \mc{W})\in B_{\mathbb C^d}(0, 1) \times \mathcal S$ and $1\le i\le d$, we have that 
\begin{equation}\label{fd}
(D_i G(\theta, \mc{W}))_\om=\mathcal L_{\sigma^{-1}\om} \left(g^i(\sigma^{-1}\om, \cdot)e^{\theta \cdot g(\sigma^{-1}\om, \cdot)}(\mc{W}_{\sigma^{-1} \om}+v_{\sigma^{-1} \om}^0)\right),
\end{equation}
for  $\om \in \Om$. 
\end{lemma}

\begin{proof}
Let us denote the right hand side in~\eqref{fd} by $(L(\theta, \mc{W}))_\om$. Furthermore, let $\{e_1, \ldots, e_n\}$ be the canonical base of $\C^d$. Observe that 
\begin{align*}
 & (G(\theta +te_i, \mc W)-G(\theta, \mc W)-tL(\theta, \mc W))_\om \displaybreak[0] \\
 &=  \mathcal  L_{\sigma^{-1} \om}((e^{(\theta+te_i)\cdot  g(\sigma^{-1} \om ,\cdot)}-e^{\theta \cdot g(\sigma^{-1} \om ,\cdot)}-tg^i(\sigma^{-1}\om ,\cdot)e^{\theta \cdot g(\sigma^{-1} \om ,\cdot)})(\mc W_{\sigma^{-1} \om}+v^0_{\sigma^{-1} \om})),
 \end{align*}
and therefore
\begin{align*}
 &\lVert (G(\theta +t, \mc W)-G(\theta, \mc W)-tL(\theta, \mc W))_\om\rVert_{\BV} \displaybreak[0] \\
 & \le K\lVert (e^{(\theta+te_i)\cdot g(\sigma^{-1} \om ,\cdot)}-e^{\theta \cdot g(\sigma^{-1} \om ,\cdot)}-tg^i(\sigma^{-1}\om ,\cdot)e^{\theta \cdot g(\sigma^{-1} \om ,\cdot)})(\mc W_{\sigma^{-1} \om}+v^0_{\sigma^{-1} \om})
  \rVert_{\BV} \displaybreak[0] \\
   & =K\var \big((e^{(\theta+te_i)\cdot  g(\sigma^{-1} \om ,\cdot)}-e^{\theta \cdot g(\sigma^{-1} \om ,\cdot)}-tg^i(\sigma^{-1}\om ,\cdot)e^{\theta g(\sigma^{-1} \om ,\cdot)})(\mc W_{\sigma^{-1} \om}+v^0_{\sigma^{-1} \om}) \big)
  \displaybreak[0] \\
  &\phantom{=}   +K \lVert (e^{(\theta+te_i) \cdot g(\sigma^{-1} \om ,\cdot)}-e^{\theta \cdot g(\sigma^{-1} \om ,\cdot)}-tg^i(\sigma^{-1}\om ,\cdot)e^{\theta \cdot  g(\sigma^{-1} \om ,\cdot)})(\mc W_{\sigma^{-1} \om}+v^0_{\sigma^{-1} \om})\rVert_1.
\end{align*}
By applying Taylor's reminder theorem for the map $z\mapsto e^{zg^i (\sigma^{-1}\om, x)}$, we obtain that 
\[
\lVert e^{tg^i (\sigma^{-1}\om, \cdot)}-1-tg^i (\sigma^{-1}\om, \cdot)\rVert_{L^\infty}\le \frac{1}{2}M^2 e^{M}\lvert t\rvert^2,
\]
and thus
\[
\lVert (e^{(\theta+te_i) \cdot g(\sigma^{-1} \om ,\cdot)}-e^{\theta \cdot g(\sigma^{-1} \om ,\cdot)}-tg^i(\sigma^{-1}\om ,\cdot)e^{\theta \cdot  g(\sigma^{-1} \om ,\cdot)} \rVert_{L^\infty}\le \frac{1}{2}M^2 e^{2M}\lvert t\rvert^2.
\]
Moreover, by applying (V9) for $f=g^i (\sigma^{-1}\om, \cdot)$ and 
$
h(z)=e^{tz}-1-tz, 
$
we conclude that for some $C>0$ independent on $\om$ and $t$,
\[
\var (e^{tg^i (\sigma^{-1}\om, \cdot)}-1-tg^i (\sigma^{-1}\om, \cdot))\le C\lvert t\rvert^2.
\]
Hence, 
\[
\var ( (e^{(\theta+te_i) \cdot g(\sigma^{-1} \om ,\cdot)}-e^{\theta \cdot g(\sigma^{-1} \om ,\cdot)}-tg^i(\sigma^{-1}\om ,\cdot)e^{\theta \cdot  g(\sigma^{-1} \om ,\cdot)})\le C'\lvert t\rvert^2, 
\]
for some $C'>0$. Now one can easily conclude that
\[
\frac{1}{\lvert t\rvert}\lVert G(\theta +te_i, \mc W)-G(\theta, \mc W)-tL(\theta, \mc W)\rVert_\infty \to 0 \quad \text{as $t\to 0$,}
\]
which yields~\eqref{fd}.
\end{proof}
The following lemma can be obtained by the same reasoning as the previous one. 
\begin{lemma}\label{l2}
For $(\theta, \mc{W})\in B_{\mathbb C^d}(0, 1) \times \mathcal S$ and $1\le i\le d$, we have that 
\begin{equation}\label{fdc}
(D_i H(\theta, \mc{W})) (\om)= \int g^i(\sigma^{-1}\om, \cdot)e^{\theta \cdot g(\sigma^{-1}\om, \cdot)}(\mc{W}_{\sigma^{-1} \om}+v_{\sigma^{-1} \om}^0)\, dm,
\end{equation}
for $\om \in \Om$.
\end{lemma}

As a direct consequence of previous lemmas we obtain the following result. 
\begin{proposition}\label{difF}
Let $F(\te,\mc W)$ be defined by (\ref{defF}).
For $(\theta, \mc W)$  in  a neighborhood $(0, 0)\in \mathbb C^d \times \mc{S}$, we have that 
 \[
  (D_{d+1} F(\theta, \mc W) \mc H)_\om=\frac{1}{H(\theta, \mc W)(\om)}\mathcal L_{\sigma^{-1} \om}^\theta \mc H_{\sigma^{-1} \om}-\frac{\int \mathcal L_{\sigma^{-1} \om}^\theta
  \mc H_{\sigma^{-1} \om}\, dm}{[H(\theta, \mc W)(\om)]^2}G(\theta, \mc W)_\om-\mc H_\om,
 \]
for $\om \in \Om$ and $\mc H \in \mathcal S$ and
\[
 \begin{split}
 (D_i F(\theta, \mc W))_\om &=\frac{1}{H(\theta, \mc W)(\om)}\mathcal L_{\sigma^{-1} \om}(g^i(\sigma^{-1} \om, \cdot)e^{\theta  \cdot g(\sigma^{-1} \om, \cdot)} (\mc W_{\sigma^{-1} \om}+
 v_{\sigma^{-1} \om}^0)) \\
 &\phantom{=}-\frac{\int g^i(\sigma^{-1} \om, \cdot)e^{\theta \cdot  g(\sigma^{-1} \om, \cdot)} (\mc W_{\sigma^{-1} \om}+
 v_{\sigma^{-1} \om}^0)\, dm}{[H(\theta, \mc W)(\om)]^2}\mathcal L_{\sigma^{-1} \om}^\theta (\mc W_{\sigma^{-1} \om}+v_{\sigma^{-1} \om}^0),
 \end{split}
\]
for $\om \in \Om$ and $1\le i\le d$. 
\end{proposition}

\begin{lemma}
We have that $D_{d+1,d+1}G=0$ and $D_{d+1,d+1}H=0$. 
\end{lemma}
\begin{proof}
The desired conclusion follows directly from Lemma~\ref{l1}.
\end{proof}
The proof of the following lemma can be obtain by repeating the  arguments from~\cite[Appendix B.2]{DFGTV1}.
\begin{lemma}\label{laux}
For $(\theta, \mc{W})\in B_{\mathbb C^d}(0, 1) \times \mathcal S$ and $i, j\in \{1, \ldots, d\}$, we have that 
\[
(D_{ji} G(\theta, \mc{W}))_\om=\mathcal L_{\sigma^{-1}\om} (g^i(\sigma^{-1}\om, \cdot) g^j (\sigma^{-1}\om, \cdot) e^{\theta \cdot g(\sigma^{-1}\om, \cdot)}(\mc{W}_{\sigma^{-1} \om}+v_{\sigma^{-1} \om}^0)),
\]
and
\[
(D_{ji} H(\theta, \mc{W})) (\om)= \int g^i(\sigma^{-1}\om, \cdot) g^j (\sigma^{-1}\om, \cdot) e^{\theta \cdot g(\sigma^{-1}\om, \cdot)}(\mc{W}_{\sigma^{-1} \om}+v_{\sigma^{-1} \om}^0)\, dm,
\]
for $\om \in \Om$. Moreover,  for $j\in \{1, \ldots, d\}$ we have that 
\[
(D_{j,d+1}G(\theta, \mc{W})\mc{H})_\om=\mathcal L_{\sigma^{-1}\om} (g^j(\sigma^{-1}\om, \cdot) e^{\theta \cdot g(\sigma^{-1}\om, \cdot)} \mc{H}_{\sigma^{-1} \om}) \quad \text{for $\mc{H}\in \mc{S}$ and $\om \in \Om$,}
\]
and
\[
(D_{j,d+1}H(\theta, \mc{W})\mc{H}) (\om)=\int g^j(\sigma^{-1}\om, \cdot) e^{\theta \cdot g(\sigma^{-1}\om, \cdot)} \mc{H}_{\sigma^{-1} \om}\, dm\quad \text{for $\mc{H}\in \mc{S}$ and $\om \in \Om$,}
\]
Finally,  for $i\in \{1, \ldots, d\}$ we have that 
\[
(D_{d+1,i} G(\theta, \mc{W}) \mathcal H)_\om=\mathcal L_{\sigma^{-1}\om} (g^i(\sigma^{-1}\om, \cdot)e^{\theta \cdot g(\sigma^{-1}\om, \cdot)}\mc{H}_{\sigma^{-1} \om}) \quad \text{for $\mc{H}\in \mc{S}$ and $\om \in \Om$,}
\]
and
\[
(D_{d+1,i} H(\theta, \mc{W}) \mathcal H) (\om)=\int g^i(\sigma^{-1}\om, \cdot)e^{\theta \cdot g(\sigma^{-1}\om, \cdot)}\mc{H}_{\sigma^{-1} \om}\, dm \quad \text{for $\mc{H}\in \mc{S}$ and $\om \in \Om$.}
\]
\end{lemma}

\subsection{A local version of the G\"artner-Ellis theorem}
Let $d\geq 1$ be an integer and let $S_n$ be a sequence of $\mathbb R^d$-valued random vectors satisfying the following condition:
\begin{assumption}\label{LD-AS}
There exists an open set $U\subset\mathbb R^d$ around the origin so that for any $t\in U$ the limit
\[
\Lambda(t):=\lim_{n\to\infty}\frac 1n\log \mathbb \E[e^{t\cdot S_n}]
\]
exists. Moreover, the function $t\mapsto \Lambda(t)$ is of class $C^2$ on  $U$, the Hessian of $\Lambda$ is positive definite at $t=0$ and $\nabla\Lambda(0)=0$.
\end{assumption}

Next, let $B\subset\mathbb R^d$ be a closed ball around the origin so $D^2 \Lambda(t)$ is positive definite for any $t\in B$, where $D^2 \Lambda (t)$ denotes the Hessian of $\Lam$ in $t$.  Consider the function $\Lambda^*:\mathbb R^d\to\mathbb R$ given by
\[
\Lambda^*(x)=\sup_{t\in B}\left(t\cdot x-\Lambda(t)\right).
\]
Then $\Lambda^*$ is a continuous convex function (continuity follows from compactness of $B$). By taking $t=0$ we see that $\Lambda^*(x)\geq0$. By considering the point $t=\delta x/\lvert x\rvert$, for some sufficiently small $\del>0$ (which depends only on $B$) and taking into account that $\Lambda$ is bounded we see that 
\[
\Lambda^*(x)\geq \delta|x|-C\to\infty\,\text{ as }\,|x|\to\infty. 
\]
In particular, using the terminology in~\cite{DemZet}, we have that $\Lambda^*$ is a good-rate function.

Our main result here is the following theorem.
\begin{thm}\label{GEThm} 
(i) For any closed set $A\subset\mathbb R^d$ we have
\[
\limsup_{n\to\infty}\frac1{n}\log \mathbb P(S_n/n\in A)\leq-\inf_{x\in A}\Lambda^*(x).
\]

(ii) There exists a closed ball $B_0$ around the origin so that for any open subset $A$ of $B_0$ we have
\[
\liminf_{n\to\infty}\frac1{n}\log \mathbb  P(S_n/n\in A)\geq-\inf_{x\in A}\Lambda^*(x).
\]
\end{thm}
\begin{rmk}
The proof of the theorem is a modification of the proof of Theorem 2.3.6 in \cite{DemZet}. We do not consider this theorem as a new result, but we have not managed to find a formulation of it in the literature. For readers' convenience we include here a complete proof.
\end{rmk}
\begin{proof}[Proof of Theorem \ref{GEThm}]
Set $\bar S_n=\frac 1n S_n$.
Let us start by establishing the upper bound. 
For any $x\in\mathbb R^d$, choose  $t(x)\in B$ such that
\[
\Lambda^*(x)=x\cdot t(x)-\Lambda(t(x)).
\]
Let $A$ be a compact subset of $\mathbb R^d$ and  take an arbitrary  $\epsilon>0$. For any $x\in A$, let $B_{x,\epsilon}$ be a ball around $x$ of radius $\epsilon$. Then
\[
\mathbb P(\bar S_n\in B_{x,\epsilon})=\mathbb E[\mathbb I(\bar S_n\in B_{x,\epsilon})]\leq \mathbb E[e^{n(\bar S_n\cdot t(x)-\inf_{y\in B_{x,\epsilon}}y\cdot t(x))}].
\]
Observe that
\[
x\cdot t(x)-\inf_{y\in B{x, \epsilon}} y\cdot t(x)=\sup_{y\in B{x,\epsilon}}\big(x\cdot t(x)-y\cdot t(x)\big)\leq \epsilon R
\]
where $R=\sup_{t\in B}|t|$. We conclude that 
\[
\limsup_{n\to\infty}\frac1{n}\log \mathbb P(\bar S_n\in B_{x,\epsilon})\leq \epsilon R-\big(x\cdot t(x)-\Lambda(t(x))\big)=
\epsilon R-\Lambda^*(x).
\]
Since $A$ is compact, we can cover it with  $N$ balls $B_{x_i,\epsilon},\,i \in \{1,2,\ldots ,N\}$ for some  $N\in \mathbb N$ and $x_1,\ldots ,x_N \in A$. Then, we have that 
\[
\log \mathbb  P(\bar S_n\in A)\leq \log N+\max_i\log \mathbb  P(\bar S_n  \in B_{x_i,\epsilon})
\]
and hence
\[
\limsup_{n\to\infty}\frac1{n}\log \mathbb  P(\bar S_n\in A)\leq R\epsilon-\min_{i}\Lambda^*(x_i).
\]
Since $\Lambda^*$ is continuous, passing to the limit when $\epsilon \to 0$,  we obtain that $\min_{i}\Lambda^*(x_i)$ converges to $\inf_{x\in A}\Lambda^*(x)$, which completes the proof of the upper bound for compact sets. In general (see~\cite[Lemma 1.2.8]{DemZet}), in order to prove the upper bound for  closed sets it is enough to prove it for compact sets and to show that the sequence $\mu_n$ of the laws of $\bar S_n$ is exponentially tight, i.e. that for any $M>0$ there is a compact set $K_M$ such that 
\begin{equation}\label{EXTIGH}
\limsup_{n\to\infty}\frac 1n\log \mu_n(\mathbb R^d\setminus K_M)<-M.
\end{equation}
Let $1\leq j\leq d$, $\rho>0$ and denote by $\bar S_{n,j}$ the $j$-th coordinate of $\bar S_{n,j}$. Denote also by $e_j$ the standard $j$-th unit vector. Let $t>0$ be sufficiently small so that $te_j\in U$. Then by the Markov inequality,
\[
\mathbb P(\bar S_{n,j}\geq\rho)\leq \mathbb  P(e^{t S_{n,j}}\geq e^{nt\rho})\leq e^{-t\rho n} \mathbb E[e^{t S_{n,j}}].
\]
Therefore,
\[
\limsup_{n\to\infty}\frac 1n\log \mathbb  P(\bar S_{n,j}\geq\rho)\leq \Lambda(te_j)-t\rho\to-\infty \text{ as }\rho\to\infty.
\]
Similarly,
\[
\lim_{\rho\to\infty}\limsup_{n\to\infty}\frac 1n\log \mathbb  P(\bar S_{n,j}\leq -\rho)=-\infty,
\]
and thus (\ref{EXTIGH}) follows.

In order to establish the lower bound, we will first observe that 
by the open mapping theorem we have the function $\nabla\Lambda$ maps the interior $B^o$ of $B$ onto an open set $V$.  Therefore, there exists an open set $V$ around the origin such that for any $y\in V$, there exists a unique $\eta=\eta(y)\in B^o$ so that $y=\nabla\Lambda(\eta)$. Since $D^2 \Lambda$ is positive definite on $B^o$,  we derive that 
\begin{equation}\label{2.3.10}
\Lambda^*(y)=y\cdot\eta-\Lambda(\eta).
\end{equation}

Next, notice that  for the lower bound to hold true for open subsets of $V$ it is enough to show that for any $y\in V$ we have
\[
\lim_{\delta\to0}\liminf_{n\to\infty}\frac1n\log \mathbb  P(\bar S_n\in B_{y,\delta})\geq-\Lambda^*(y).
\]
Let $y\in V$ and write $y=\nabla \Lam (\eta)$. We will make now an exponential change of measure: consider the probability measures $\tilde\mu_n$ given by 
\[
\frac{d\tilde\mu_n}{d\mu_n}(z)=\exp(n\eta\cdot z-\Lambda_n(\eta))
\] 
where $\mu_n$ is the law of $\bar S_n$ and $\Lambda_n(\eta)=\log \mathbb  E[e^{\eta\cdot S_n}]$. Let $\bar Z_n$ be a random vector distributed according to $\tilde\mu_n$ and set $Z_n=n\bar Z_n$. Then
\begin{eqnarray*}
\frac1n\log  \mu_n(B_{y,\delta})=\frac1n\log \mathbb E[e^{\eta\cdot S_n}]-\eta\cdot y+\frac1n\log \int_{z\in B_{y,\delta}}e^{n\eta\cdot(y-z)}d\tilde\mu_n(z)\\\geq \frac1n\log \mathbb E[e^{\eta\cdot S_n}]-\eta\cdot y-|\eta|\delta+\frac1n\log \tilde\mu_n(B_{y,\delta})
\end{eqnarray*}
and therefore
\begin{eqnarray*}
\lim_{\delta\to0}\liminf_{n\to\infty}\frac 1n\log \mu_n(B_{y,\del})\geq\Lambda(\eta)-\eta\cdot y+
\lim_{\delta\to0}\liminf_{n\to\infty}\frac 1n\log \tilde\mu_n(B_{y,\del})\\\geq
-\Lambda^*(y)+\lim_{\delta\to0}\liminf_{n\to\infty}\frac 1n\log \tilde\mu_n(B_{y,\del}).
\end{eqnarray*}

The proof of the lower bound will be complete once we show that for any $\delta>0$,
\begin{equation}\label{Once}
\lim_{\delta\to 0}\liminf_{n\to\infty}\frac 1n\log \tilde\mu_n(B_{y,\delta})=0.
\end{equation}
Define $T_n=Z_n-ny$. Then for any $t\in U_\eta:=U-\{\eta\}$ we have
\[
\boldsymbol{\Lambda}(t):=\lim_{n\to\infty}\frac 1n\log \mathbb  E[e^{t\cdot T_n}]=\Lambda(t+\eta)-\Lambda(\eta)-t\cdot y.
\]
Observe next that $\boldsymbol{\Lambda}$ satisfies all the conditions in Assumption \ref{LD-AS} with $U_\eta$ in place of $U$. Therefore, setting
\begin{equation}\label{StarForm}
\boldsymbol{\Lambda}^*(x):=\sup_{t\in B_\eta}\big(x\cdot t-\boldsymbol{\Lambda}(t)\big)=\Lambda^*(x+y)+\Lambda(\eta)-(x+y)\cdot\eta ,
\end{equation}
using the already established  large deviations upper bound we obtain that
\[
\limsup_{n\to\infty} \frac 1 n \log \tilde\mu_n(\mathbb R^d\setminus B_{y,\delta})=
\limsup_{n\to\infty}\frac 1 n \log \mathbb P(|T_n/n|\geq \del)\leq-\inf_{|z|\geq\del}\boldsymbol{\Lambda}^*(z)=-\boldsymbol{\Lambda}^*(z_0),
\]
for $z_0$ such that $|z_0|\geq \delta$ (observe that the existence of $z_0$ follows from $\lim_{|z|\to\infty}\boldsymbol{\Lambda}^*(z)=\infty$). We claim that $\boldsymbol{\Lambda}^*(z_0)>0$. Using this we obtain that
\[
\limsup_{n\to\infty} \frac 1 n \log \tilde\mu_n(\mathbb R^d\setminus B_{y,\delta})<0,
\]
for any $\delta>0$. Hence, $\tilde\mu_n(\mathbb R^d\setminus B_{y,\delta}) \to 0$ and therefore $\tilde\mu_n( B_{y,\delta}) \to 1$ for every $\delta >0$
which clearly implies (\ref{Once}). 

Let us now show that $\boldsymbol{\Lambda}^*(z_0)>0$. Assume the contrary, i.e. that  $\boldsymbol{\Lambda}^*(z_0)=0$. Then by (\ref{StarForm}) we have
\begin{equation}\label{tt}
\Lambda^*(y+z_0)=\sup_{t\in B}\big(t\cdot (y+z_0)-\Lambda(t)\big)=\eta \cdot (y+z_0)-\Lambda(\eta).
\end{equation}
This means that the supremum in~\eqref{tt} is actually maximum and it is  achieved at $t=\eta$ and so 
$y=\nabla\Lambda(\eta)=y+z_0$ which is a contraction since $z_0\not=0$.
\end{proof}

\begin{rmk}
It is clear from the proof of Theorem \ref{GEThm} that if for some positive sequence $(\ve_n)_n$ so that $\lim_{n\to\infty}\ve_n=0$ the limit
\[
\Lambda(t):=\lim_{n\to\infty}\ve_n\log \mathbb \E[e^{t\cdot S_n}]
\]
exists in some neighborhood $U$ of the origin, and satisfies all the other conditions required in Assumption \ref{LD-AS}, then: 

(i) For any closed set $A\subset\mathbb R^d$ we have
\[
\limsup_{n\to\infty}\ve_n\log \mathbb P(\ve_n S_n\in A)\leq-\inf_{x\in A}\Lambda^*(x).
\]

(ii) There exists a closed ball $B_0$ around the origin so that for any open subset $A$ of $B_0$ we have
\[
\liminf_{n\to\infty}\ve_n\log \mathbb  P(\ve_n S_n\in A)\geq-\inf_{x\in A}\Lambda^*(x).
\]
\end{rmk}

\bibliographystyle{abbrv}

\end{document}